\documentclass{article}

\usepackage{amsthm}
\usepackage[english]{babel}
\usepackage{amsmath}
\usepackage{amssymb}
\usepackage[%
]{hyperref}
\usepackage{booktabs}
\usepackage{enumerate}
\usepackage[left = 25mm, top = 40mm, bottom = 35mm, right = 25mm]{geometry}
\usepackage{multicol}
\usepackage{bbm}
\usepackage[]{todonotes}
\usepackage{setspace}
\usepackage{mathtools}
\usepackage {url}
\usepackage{xcolor}

\theoremstyle{definition}
\newtheorem{definition}{Definition}
\newtheorem{Remark}[definition]{Remark}
\newtheorem{Example}[definition]{Example}

\theoremstyle{plain}
\newtheorem{theorem}[definition]{Theorem} 
\newtheorem{lemma}[definition]{Lemma}

\newtheorem{Corollary}[definition]{Corollary}

\newtheorem*{Question*}{Question}

\newtheorem*{theorem*}{Theorem}
\newtheorem{theoremintro}{Theorem}

\DeclareMathOperator{\Aut}{Aut}
\DeclareMathOperator{\GL}{GL}
\DeclareMathOperator{\AGL}{AGL}
\DeclareMathOperator{\PSL}{PSL}
\DeclareMathOperator{\PSU}{PSU}

\DeclareMathOperator{\SL}{SL}
\DeclareMathOperator{\Out}{Out}

\DeclareMathOperator{\Syl}{Syl}
\DeclareMathOperator{\ord}{ord}

\newcommand{\N}{\mathbb{N}}
\newcommand{\Z}{\mathbb{Z}}
\newcommand{\F}{\mathbb{F}}
\newcommand{\BP}{\mathbb{P}}
\renewcommand{\Im}{\operatorname{Im}}
\newcommand{\Ker}{\operatorname{Ker}}

\newcommand{\ug}{\mathcal{U}(G)}

\newcommand{\op}{O_{p}(G)}

\newcommand{\fo}{F_0(G)}

\newcommand\extrafootertext[1]{%
	\bgroup
	\renewcommand\thefootnote{\fnsymbol{footnote}}%
	\renewcommand\thempfootnote{\fnsymbol{mpfootnote}}%
	\footnotetext[0]{#1}%
	\egroup
}

\numberwithin{definition}{section}
\numberwithin{equation}{section}

\title{Tuple regularity and $k$-ultrahomogeneity for finite groups}
\author{\textbf{Sofia Brenner} \\
	\normalsize\emph{TU Darmstadt, Germany}\\
	\normalsize{\texttt{sofia.brenner@tu-darmstadt.de}}}
\date{\vspace{-0.5cm}}

\begin{document}
\maketitle
\begin{abstract}
\noindent For $k, \ell \in \N$, we introduce the concepts of $k$-ultrahomogeneity and $\ell$-tuple regularity for finite groups. Inspired by analogous concepts in graph theory, these form a natural generalization of homogeneity, which was studied by Cherlin and Felgner~\cite{CHE91, CHE00} and Li~\cite{Li99} as well as automorphism transitivity, which was investigated by Zhang~\cite{ZHA92}. Additionally, these groups have an interesting algorithmic interpretation. We classify the $k$-ultrahomogeneous and $\ell$-tuple regular finite groups for $k, \ell \geq 2$. In particular, we show that every 2-tuple regular finite group is ultrahomogeneous.
\end{abstract}

\section{Introduction}	
Recently, there has been interest in studying finite groups from a combinatorial perspective, leading to the transfer of fundamental tools from graph theory to the realm of finite groups. For instance, versions of the Weisfeiler-Leman algorithm, a combinatorial algorithm to investigate whether two input graphs are isomorphic, have been proposed for finite groups \cite{BRA20}. This motivates the study of combinatorial symmetry and regularity measures for finite groups. 
\medskip

In this paper, we introduce the concepts of $k$-ultrahomogeneity and $\ell$-tuple regularity (for $k,\ell \in \N)$ for finite groups. Inspired by classical graph theoretic concepts, these notions capture different combinatorial aspects of symmetry while, at the same time, generalizing previously known group theoretic concepts. A finite group is called $k$-ultrahomogeneous if every isomorphism between two $k$-generated subgroups extends to an automorphism of the entire group, and ultrahomogeneous if it is $k$-ultrahomogeneous for all $k \in \N$ (note that some authors use the term ``homogeneity'' -- we adopt the terminology ``ultrahomogeneity'' to distinguish our definition of $k$-ultrahomogeneity from the unrelated notion of $k$-homogeneity defined for permutation groups). Complementary to this, we introduce the concept of $\ell$-tuple regularity. Intuitively speaking, a group is $\ell$-tuple regular if for all pairs of isomorphic $\ell$-generated subgroups, the multisets of subgroups obtained by adding an extra generator in all possible ways are the same. 
\medskip


The notion of $k$-ultrahomogeneity generalizes the well-known concepts of ultrahomogeneity and automorphism transitivity. Ultrahomogeneous groups are interesting from a model theoretic point of view. The finite ultrahomogeneous groups were classified by Cherlin and Felgner \cite{CHE91, CHE00} and, independently, by Li~\cite{Li99}. The work of Cherlin and Felgner additionally contains several results on infinite solvable ultrahomogeneous groups. 
Automorphism transitive groups are finite groups in which any two elements of the same order can be mapped to each other by a group automorphism. In our framework, these are precisely the 1-ultrahomogeneous groups.
Zhang~\cite{ZHA92} described possibilities for their structure. However, it has not been determined which of these groups are indeed automorphism transitive. Another related extensively studied concept that generalizes automorphism transitive groups are $m$-DCI-groups (see, for instance,~\cite{LI98}).
Studying $\ell$-tuple regularity is motivated by an algorithmic perspective as $\ell$-tuple regular groups are closely related to groups for which the Weisfeiler-Leman algorithm terminates after the first iteration (see Section~\ref{sec:preliminaries}).
\medskip

For every $k,\ell \geq 2$, we study the class of $k$-ultrahomogeneous finite groups as well as the class of $\ell$-tuple regular finite groups. By definition, each of these contains the class of ultrahomogeneous groups. Moreover, every $k$-ultrahomogeneous group is $k$-tuple regular, and every $\ell$-tuple regular group is $\ell'$-tuple regular for every $\ell' \leq \ell$. For this reason, we first study the class of 2-tuple regular finite groups, which is a priori the largest of the above-mentioned classes. There, we derive the following classification result: 

\begin{theoremintro}\label{theo:a}
Let $G$ be a finite group. 
\begin{enumerate}[(i)]
\item If $G$ is solvable, then $G$ is 2-tuple regular if and only if $G = A \times B$, where $A$ and $B$ have coprime orders, $A$ is an abelian group with homocyclic Sylow subgroups, and one of the following holds for $B$: 
\begin{enumerate}[(a)]
\item $B$ is isomorphic to one of the groups in $\{1, Q_8, G_{64}, A_4, C_3^2 \rtimes Q_8, \SL(2,3), G_{192}\}$, where $G_{64}$ denotes the Suzuki 2-group of order 64 and $G_{192}$ is a certain group of order 192. 
\item $B \cong M \rtimes C_{2^n}$, where $M$ is an abelian group of odd order with homocyclic Sylow subgroups and the cyclic group acts on $M$ by inversion. 
\end{enumerate} 
\item If $G$ is non-solvable, then $G$ is 2-tuple regular if and only if $G = H \times E$, where $H$ and $E$ have coprime orders, $H$ is an abelian group with homocyclic Sylow subgroups and $E$ is isomorphic to one of $\SL(2,5)$, $\PSL(2,5)$, and $\PSL(2,7).$
\end{enumerate}
\end{theoremintro}	

The proof of this theorem uses the classification of the finite simple groups. Now by comparing the classification given in Theorem~\ref{theo:a} to the classification of ultrahomogeneous finite groups in \cite{CHE00}, we obtain the following equalities between the group classes introduced above: 

\begin{theoremintro}\label{theo:b}
Let $G$ be a finite group. Then the following are equivalent: 
\begin{enumerate}[(i)]
\item $G$ is $k$-ultrahomogeneous for some $k \geq 2$. 
\item $G$ is $\ell$-tuple regular for some $\ell \geq 2$. 
\item $G$ is ultrahomogeneous. 
\end{enumerate}
\end{theoremintro}

From an algorithmic perspective, this result is interesting as it relates ultrahomogeneity, a symmetry condition, to 2-tuple regularity, a property only involving 2- and 3-generated subgroups that can be checked without knowledge of the entire automorphism group. It provides us with a simple algorithm to determine whether a given group is ultrahomogeneous: it suffices to check whether the Weisfeiler-Leman algorithm stabilizes after the first round. In fact, a similar phenomenon occurs for graphs: Cameron~\cite{CAM80} proved that every $5$-tuple regular graph is ultrahomogeneous, so the classes of $k$-ultrahomogeneous and $\ell$-tuple regular graphs coincide for every $k , \ell \geq 5$. As for groups, there is no direct argument for this equivalence known, but it is proven by comparing the respective graph classes. 
\medskip

Combined with Theorem~\ref{theo:b}, Theorem~\ref{theo:a} settles the classification of the $k$-ultrahomogeneous and $\ell$-tuple regular finite groups for every $k, \ell \geq 2$. In contrast to this, we remark that there exist 1-ultrahomogeneous finite groups that are not ultrahomogeneous (see \cite{ZHA92}). In other words, the class of 1-ultrahomogeneous finite groups is larger than the above-mentioned class. 
\medskip

This paper is organized in the following way: In Section~\ref{sec:preliminaries}, we introduce our notation and recall some preliminary results. In Section~\ref{sec:generalconcepts}, we prove several results of general flavor on $k$-ultrahomogeneity and $\ell$-tuple regularity. In Section~\ref{sec:pgroups}, we classify the $2$-tuple regular $p$-groups, thereby mainly using results on their power structure. In Section~\ref{sec:solvable}, we classify the solvable 2-tuple regular finite groups. Building on these results, we prove the classification of 2-tuple regular finite groups stated in Theorems~\ref{theo:a} as well as Theorem~\ref{theo:b} in Section~\ref{sec:general}. 

\section{Preliminaries}\label{sec:preliminaries}
Let $G$ be a finite group. We use the standard group-theoretic terminology. 
We write $[x]_G$ for the $G$-conjugacy class of $x \in G$ and omit the index if the ambient group is clear. 
For $a,b \in G$, we define $[a,b] \coloneqq aba^{-1}b^{-1}$, and for $A,B \subseteq G$, we set $[A,B] \coloneqq \langle [a,b] \colon a  \in A, b \in B \rangle$ as usual. Let $\gamma_1(G) \coloneqq G$ and $\gamma_{i+1}(G) \coloneqq [\gamma_i(G),G]$ for $i \in \N$. 
Set \[F_0(G) \coloneqq \prod_{p \in \mathbb{P} \setminus \{2\}} O_p(G).\] By $d(G)$, we denote the minimal number of generators of $G$. The quasisimple subnormal subgroups of $G$ are called the components of $G$. We write $E(G)$ for the layer of $G$, that is, the subgroup generated by the components of $G$.
\medskip

As usual, we denote by $Q_{2^k}$ the generalized quaternion group of order~$2^k$ for $k \geq 3$. For $d \in \N$ and $p \in \BP$, let $\AGL(1,p^d)$ and $\operatorname{\Gamma L}(1, p^d)$ denote the 1-dimensional affine linear group and semilinear group over the field with $p^d$ elements, respectively. Let $G_{64}$ be the Suzuki 2-group of order 64, which can be realized as a Sylow 2-subgroup of $\PSU(3,4)$. Moreover, let $G_{192} \cong G_{64} \rtimes C_3$ denote the group $\texttt{SmallGroup}(192,1025)$ in the \texttt{SmallGroupsLibrary} in GAP~\cite{GAP4}. Finally, $2.A_7$ and $6.A_7$ denote the double and the sextuple cover of the alternating group $A_7$. 
\medskip

Now suppose that $G$ is a finite $p$-group for some prime number $p \in \BP$. The group $G$ is called homocyclic if it is isomorphic to a direct product of $k$ copies of the cyclic group $C_{p^\ell}$ for some $k, \ell \in \N$. For $i \in \N_0$, let $\Omega_i \coloneqq \Omega_i(G)$ be the subgroup generated by all elements of order dividing $p^i$ in $G$ and write $\Omega \coloneqq \Omega_1$. Moreover, we set $\mho^i \coloneqq \mho^i(G)$ to be the subgroup generated by the $p^i$-th powers of the elements in $G$ and set $\mho \coloneqq \mho^1$. The group $G$ is powerful if $p$ is odd and $G' \subseteq \mho$ holds, or if $p = 2$ and $G' \subseteq \mho^2$ holds. 
\medskip

We now define the concepts of $k$-ultrahomogeneity and $\ell$-tuple regularity studied in this paper. 

\begin{definition}[$k$-ultrahomogeneity]
Let $G$ be a finite group, and let $k \in \N$. The group $G$ is called $k$-\emph{ultraho\-mo\-geneous} if the following holds: For all tuples $(g_1, \ldots, g_k), (h_1, \ldots, h_k) \in G^k$ for which the assignment $g_1 \mapsto h_1$, \ldots, $g_k \mapsto h_k$ defines an isomorphism $\varphi$ between $\langle g_1, \ldots, g_k \rangle$ and $\langle h_1, \ldots, h_k \rangle$, there exists an automorphism $\hat{\varphi} \in \Aut(G)$ with $\hat{\varphi}|_{\langle g_1, \ldots, g_k \rangle} = \varphi$. The group $G$ is called \emph{ultrahomogeneous} if it is $k$-ultrahomogeneous for every $k \in \N$. 
\end{definition}

We remark that the last property is sometimes called homogeneity instead of ultrahomogeneity (for instance, in \cite{CHE91, CHE00}). In this paper, we use the terminology ``$k$-ultrahomogeneity'' to distinguish the property from the notion of $k$-homogeneity defined in the setting of permutation groups. 
Note that a finite group $G$ is 1-ultrahomogeneous if and only if it is automorphism transitive, that is, every two elements of the same order are conjugate in the automorphism group of $G$. The structure of automorphism transitive groups was studied in \cite{ZHA92}.

\begin{definition}[$\ell$-tuple regularity]
	Let $G$ be a finite group, and let $\ell \in \N$. The group $G$ is called $\ell$-\emph{tuple regular} if the following holds: For all $\ell$-tuples $(g_1, \ldots, g_\ell), (h_1, \ldots, h_\ell) \in G^\ell$ for which the assignment $g_1 \mapsto h_1$, \ldots, $g_\ell \mapsto h_\ell$ defines an isomorphism $\varphi$ between $\langle g_1, \ldots, g_\ell \rangle$ and $\langle h_1, \ldots, h_\ell \rangle$, there exists a bijection $\Psi \colon G \to G$ such that for every $g \in G$, the extended assignment $g_1 \mapsto h_1$, \dots, $g_\ell \mapsto h_\ell$, $g \mapsto \Psi(g)$ defines an isomorphism between $\langle g_1, \ldots, g_\ell, g\rangle$ and $\langle h_1, \ldots, h_\ell, \Psi(g) \rangle$. 
\end{definition}

In this situation, we call $\Psi$ an \emph{extending bijection}. Note that $\Psi$ preserves the order of the elements. Observe that, as the tuples considered in the definitions may contain repeated entries, every $k$-ultrahomogeneous group is $k'$-ultrahomogeneous for all $k' \leq k$ (similarly for tuple regularity). Clearly, every $k$-ultrahomogeneous finite group is $k$-tuple regular. Concerning the converse implication, we observe the following: 

\begin{lemma}\label{lemma:kgenerated}
Let $\ell \in \N$, and let $G$ be a finite $\ell$-generated $\ell$-tuple regular group. Then $G$ is 1-ultra\-homo\-geneous. 
\end{lemma}

\begin{proof}
Fix a generating set $\{g_1, \ldots, g_\ell\}$ of $G$ and let $x,y \in G$ be elements of the same order. Since $G$ is in particular 1-tuple regular, we find $g_1' \in G$ such that the assignment $x \mapsto y$, $g_1 \mapsto g_1'$ defines an isomorphism between $\langle x, g_1 \rangle$ and $\langle y, g_1' \rangle$. Continuing this way, we iteratively construct elements $g_2', \ldots, g_\ell'$ such that the assignment $x \mapsto y$, $g_1 \mapsto g_1'$, \dots, $g_\ell \mapsto g_\ell'$ defines an isomorphism between $G = \langle x, g_1, \ldots, g_\ell \rangle$ and $\langle y, g_1', \ldots, g_\ell' \rangle$, that is, an automorphism of $G$.
\end{proof}

From an algorithmic perspective, the notion of $\ell$-tuple regularity is closely linked to the number of refinement steps needed to reach the stable coloring in the Weisfeiler-Leman algorithm (using the second variant of the algorithm described in \cite{BRA20}). If a group is $\ell$-tuple regular for some $\ell \geq 2$, then the $\ell$-dimensional Weisfeiler-Leman algorithm will reach the stable coloring in the initial step. Conversely, if $G$ is a finite group for which the coloring computed by the $\ell$-dimensional Weisfeiler-Leman algorithm stabilizes in the initial step, then $G$ is $(\ell-1)$-tuple regular. 
\medskip

We conclude this section with an example showing that the ordering of the tuples is relevant: 

\begin{Example}
Consider the group $G \coloneqq \langle x,y \colon x^9 = y^9 = 1,\ y x y^{-1} = x^{7}\rangle \cong C_9 \rtimes C_9$. For every pair of elements $g_1, g_2 \in G$ of the same order, we find a bijection $\Psi \colon G \to G$ such that $\langle g_1, g \rangle \cong \langle g_2, \Psi(g)\rangle$ holds. However, one can check that $G$ is not 1-tuple regular (for instance, this follows from Remark~\ref{rem:centralizer} below). 
\end{Example}

\section{\texorpdfstring{General results on $\ell$-tuple regularity}{General results on l-tuple regularity}}\label{sec:generalconcepts}
In this section, we derive general properties of $\ell$-tuple regular finite groups. Most of them hold for arbitrary $\ell \in \N$. At the end of this section, though, we present some statements that require at least 2-tuple regularity.

\begin{Remark}\label{rem:centralizer}
Let $G$ be a $1$-tuple regular finite group and let $a,b \in G$ be elements of the same order. Consider the assignment $a \mapsto b$ and let $\Psi \colon G \to G$ be an extending bijection. 
\begin{enumerate}[(i)]
\item The map $\Psi$ induces an order-preserving bijection between $C_G(a)$ and $C_G(b)$ as well as between $N_G(\langle a \rangle)$ and $N_G(\langle b \rangle)$. In particular, the conjugacy classes of $a$ and $b$ have the same size.
\item For $x \in G$ and $\ell \in \N$, we have $x^\ell = a$ if and only if $\Psi(x)^\ell = b$ holds. In particular, if $G$ contains a maximal cyclic subgroup of order $m \in \N$, then there is no maximal cyclic subgroup of $G$ whose order is a proper divisor of $m$. 
\end{enumerate}
\end{Remark}

%
%

We say that a subgroup $N$ of $G$ is \emph{a union of order classes} if, for every $n \in \N$, it contains all or none of the elements of order $n$ in~$G$. In this case, $N$ is a normal subgroup of $G$. The set of these subgroups is denoted by $\ug$. They will play a central role in our derivation as they inherit the $\ell$-tuple regularity of $G$: 

\begin{lemma}\label{lemma:orderinn}
Let $\ell \in \N$, let $G$ be an $\ell$-tuple regular finite group and consider $N \in \ug$. Then $N$ is $\ell$-tuple regular.
\end{lemma}

\begin{proof}
	Let $n_1, \ldots, n_\ell, n_1', \ldots, n_\ell' \in N$ such that the assignment $n_1 \mapsto n_1'$, \dots, $n_\ell \mapsto n_\ell'$ defines an isomorphism between $\langle n_1, \ldots, n_\ell\rangle$ and $\langle n_1', \ldots, n_\ell' \rangle$. Let $\Psi \colon G \to G$ be a corresponding extending bijection. Since $\Psi$ is order-preserving, we have $n \in N$ precisely if $\Psi(n) \in N$ holds. Therefore, the restriction of $\Psi$ to $N$ is an extending bijection in $N$.
\end{proof}

The following statement is an easy consequence of the preceding result:

\begin{Corollary}\label{cor:directproduct}
	Let $G = G_1 \times \dots \times G_n$ be a finite group with subgroups $G_1, \ldots, G_n$ of pairwise coprime order. For $\ell \in \N$, the group $G$ is $\ell$-tuple regular if and only if $G_1, \ldots, G_n$ are $\ell$-tuple regular. 
\end{Corollary}

The following result will be needed frequently:

\begin{lemma}\label{lemma:centralizerunionorderclasses}
Let $G$ be a $1$-tuple regular finite group. For $N \in \ug$, we have $C_G(N) \in \ug$. 
\end{lemma}

\begin{proof}
Let $x,y \in G$ be of the same order, and assume $x \in C_G(N)$, so $N \subseteq C_G(x)$ holds. Due to $N \in \ug$ and the fact that there is a bijection between the multisets of orders of the elements in $C_G(x)$ and $C_G(y)$ (see Remark~\ref{rem:centralizer}), we obtain $y \in C_G(N)$.
\end{proof}

In general, quotient groups of $\ell$-tuple regular groups are not necessarily $\ell$-tuple regular. However, we obtain the following result:

\begin{lemma}\label{lemma:preimageorder}\label{lemma:centralizersq}
	Let $G$ be a 1-tuple regular finite group, and let $N \in \ug$. Set $\bar{G} \coloneqq G/N$, and let $a,b \in \bar{G}$ be elements of the same order. 
	\begin{enumerate}[(i)]
	\item There exist preimages $g_a, g_b \in G$ with $\ord(g_a) = \ord(g_b)$ of $a$ and $b$, respectively.
	\item The conjugacy classes of $a$ and $b$ in $\bar{G}$ have the same size.
	\end{enumerate} 
\end{lemma}

\begin{proof}
$\null$
\begin{enumerate}[(i)]
	\item Let $\pi$ denote the set of prime divisors of $\ord(a) = \ord(b)$. Let $g_a$ and $g_b$ be preimages of $a$ and $b$ whose orders are only divisible by the primes in $\pi$. We claim that $o_a \coloneqq \ord(g_a) = \ord(g_b) \eqqcolon o_b$ holds. To see this, first assume $\pi = \{p\}$ for some $p \in \BP$. Without loss of generality, we may assume $o_a^k = o_b$ for some $k \in \N$. Then $g_a^k$ and $g_b$ have the same order, and due to $N \in \ug$, also $g_a^k N$ and $g_b N$ have the same order in $\bar{G}$. Hence $\gcd(p,k) = 1$ follows, which implies $o_a = o_b$. For the general case, we factor $g_a$ and $g_b$ into products of commuting elements of distinct prime power orders and apply the first part of this proof to the individual factors. 

	\item 
	By (i), there exist preimages $g_a, g_b \in G$ of $a$ and $b$ such that $\ord(g_a) = \ord(g_b)$ holds. By Remark~\ref{rem:centralizer}, $[g_a]$ and $[g_b]$ have the same size. Let $\Psi \colon G \to G$ be a bijection extending the assignment $g_a \mapsto g_b$. For $g \in G$, we have $[g,g_a] \in N$ precisely if $[\Psi(g), g_b] \in N$ holds due to $N \in \ug$. This implies that $[a]_{\bar{G}}$ and $[b]_{\bar{G}}$ have the same size as these conjugacy classes are the images of $[g_a]$ and $[g_b]$ in $\bar{G}$, respectively.  \qedhere
\end{enumerate}
\end{proof}

\begin{lemma}\label{lemma:oqfo}	
Let $\ell \in \N$ and let $G$ be an $\ell$-tuple regular finite group.
	\begin{enumerate}[(i)]
		\item For every $p \in \BP$, we have $\op \in \ug$. Moreover, we have $\fo \in \ug$ and $F(G) \in \ug$.
		\item Suppose that $\ell \geq 2$ holds. For every $p \in \BP$ and $N \in \ug$, the preimage of $O_p(G/N)$ in $G$ is contained in~$\ug$. 
	\end{enumerate}
	In particular, the above groups are $\ell$-tuple regular.
\end{lemma}

\begin{proof}
$\null$
\begin{enumerate}[(i)]
\item Fix $p \in \BP$. Let $x \in \op$ and let $\ell \in \N$ with $\ord(x) = p^\ell$. For every $x' \in [x]_G$, the subgroup $\langle x, x' \rangle$ is a $p$-group. Now consider an element $y \in G$ of order~$p^\ell$. Let $y' \in [y]_G$ and write $y' = g y g^{-1}$ for some $g \in G$. Consider the assignment $y \mapsto x$ and let $\Psi \colon G \to G$ be an extending bijection. Let $\varphi$ denote the isomorphism between $\langle y, g \rangle$ and $\langle x, \Psi(g) \rangle$ defined by $y \mapsto x, g \mapsto \Psi(g)$. Then $\varphi$ restricts to an isomorphism between $\langle y, y'\rangle$ and $\langle x, \Psi(g) x \Psi(g)^{-1} \rangle$. It follows that $\langle y, y' \rangle$ is a $p$-group. By \cite[Theorem~3.8.2]{GOR68}, this implies $y \in \op$. 
\medskip

Now let $g \in F(G) = \prod_{p \in \BP} \op$, and let $h \in G$ be an element of order $\ord(g)$. Write $g = \prod_{p \in \BP} g_p$ with $g_p \in \op$ for every $p \in \BP$. Moreover, we decompose $h$ in the form $\prod_{p \in \BP} h_p$ for $p$-elements $h_p$ with $[h_p, h_q] = 1$ for $p,q \in \BP$. 
For every $p \in \BP$, we then obtain $\ord(g_p) = \ord(h_p)$. By the first part of this proof, this implies $h_p \in F(G)$ for all $p \in \BP$, which yields $h \in F(G)$. Similarly, we argue for $F_0(G)$.
	
\item Set $\bar{G} \coloneqq G/N$ and let $H$ denote the preimage of $O_p(\bar{G})$ in $G$. Since we have \[O_p(\bar{G}) = \{z \in \bar{G} \colon \langle z, z' \rangle \text{ is a $p$-group for all }z' \in [z]_{\bar{G}}\}\] 
(see~\cite[Theorem 3.8.2]{GOR68}), we obtain 
	\[H = \{z \in G \colon \langle z, z' \rangle N/N \text{ is a $p$-group for all }z' \in [z]_{G}\}.\] 
	Let $x,y \in G$ be elements of the same order, and assume $x \in H$ and $y \notin H$. Let $y' \in [y]_G$ such that $\langle y, y' \rangle N/N$ is not a $p$-group, and write $y' = gyg^{-1}$ for some $g \in G$. Consider the assignment $y \mapsto x$ and let $\Psi \colon G \to G$ be an extending bijection. Let $\varphi \colon \langle y,g \rangle \to \langle x, \Psi(g)\rangle$ denote the isomorphism defined by $y \mapsto x, g \mapsto \Psi(g)$. For $x' \coloneqq \varphi(y')$, the restriction of $\varphi$ defines an isomorphism between $\langle y,y' \rangle$ and $\langle x,x' \rangle$. Note that we have $x' = \Psi(g) x \Psi(g)^{-1} \in [x]_G$.
	Let $a \in \langle y, y' \rangle$ such that $a N \in \bar{G}$ is not a $p$-element, and set $a' = \varphi(a) \in \langle x,x' \rangle$. By assumption, we have $c' \coloneqq a'^{p^s} \in N$ for some $s \in \N$, but $c \coloneqq a^{p^s} \notin N$. Due to $N \in \ug$, this implies $\ord(c) \neq \ord(c')$, which is a contradiction to $\varphi(c) = c'$.
\end{enumerate}	
By Lemma~\ref{lemma:orderinn}, the given subgroups of $G$ are $\ell$-tuple regular.
\end{proof}

Another central aspect distinguishing 2-tuple regularity from 1-tuple regularity is the fact that 2-tuple regularity captures commutator relations:

\begin{lemma}\label{lemma:commutators2tupleregular}
Let $G$ be a 2-tuple regular finite group and let $a,b,a',b' \in G$ such that the assignment $a \mapsto a', b \mapsto b'$ defines an isomorphism between $\langle a,b\rangle$ and $\langle a', b'\rangle$. If $b = gag^{-1}$ holds for some $g \in G$, then there exists $g' \in G$ of order $\ord(g)$ with $b' = g'a'g'^{-1}$. 
\end{lemma}

\begin{proof}
By 2-tuple regularity, there exists a bijection $\Psi \colon G \to G$ extending the assignment $a \mapsto a', b \mapsto b'$. In particular, we have $\langle a,b,g \rangle \cong \langle a', b' , \Psi(g)\rangle$, and hence $b = gag^{-1}$ implies $b' = \Psi(g) a' \Psi(g)^{-1}$. Moreover, the elements $g$ and $\Psi(g)$ have the same order.
\end{proof}

\section{Groups of prime power order}\label{sec:pgroups}
In this section, we study 2-tuple regular groups of prime power order. They form the building blocks of our derivation. In Section~\ref{sec:powerstructure}, we investigate their power structure. These results are then used in the classification given in Section~\ref{sec:classificationpgroups}. 

\subsection{Power structure}\label{sec:powerstructure}
Throughout, let $p \in \BP$ denote a prime number. In this section, we study the power structure of 2-tuple regular finite $p$-groups. However, many of our results even hold for 1-tuple regular finite $p$-groups. 

\begin{Remark}\label{rem:onetupleregular}	
Let $G$ be a 1-tuple regular finite $p$-group. By Remark~\ref{rem:centralizer}, we have $Z(G) = \Omega_j$ for some $j \in \N$. In particular, $Z(G)$ contains all elements of order $p$ in $G$. Moreover, all maximal cyclic subgroups of $G$ have the same order. 
\end{Remark}

\begin{Example}[Generalized quaternion groups]\label{ex:quaterniongroup}
	The group $Q_8$ is ultrahomogeneous by \cite[Proposition 8]{CHE91}, so $\ell$-tuple regular for every $\ell \in \N$. For $k \geq 4$, the group $Q_{2^k}$ contains maximal cyclic subgroups of different orders. In particular, it is not 1-tuple regular by Remark~\ref{rem:onetupleregular}.
\end{Example}


Remark~\ref{rem:onetupleregular} gives rise to the following fundamental result: 

\begin{lemma}\label{lemma:omegaagemo}
	Let $G$ be a $1$-tuple regular finite group of exponent $p^s$ for some $s \in \N$. For every $i \in \{0, \ldots, s\}$, we have $\Omega_{s-i} = \mho^i$. 
\end{lemma}

\begin{proof}
	Let $i \in \{0, \ldots, s\}$. The inclusion $\mho^i \subseteq \Omega_{s-i}$ holds in any finite $p$-group of exponent $p^s$. Conversely, let $a \in G$ be an element of order $p^{s-i}$. By Remark~\ref{rem:onetupleregular}, there exists an element $b \in G$ of order $p^s$ with $a = b^{p^i}$ and hence $a \in \mho^i$ follows.
\end{proof}

\begin{lemma}\label{lemma:1trpowerful}
Let $G$ be a 1-tuple regular finite group of exponent $p^s$ for some $s \in \N$. For all nontrivial elements $a,b \in G$, we have $\ord([a,b]) < \min\{\ord(a), \ord(b)\}$. In particular, if $p$ is odd, $G$ is a powerful $p$-group.
\end{lemma}

\begin{proof}
Let $a,b \in G$ be nontrivial elements and assume that $c \coloneqq [a,b]$ has order $\ell \geq \ord(a)$. Then there exists $d \in G'$ with $\ord(d) = \ord(a)$. Let $\Psi \colon G \to G$ be an extending bijection for the assignment $a \mapsto d$. By definition, we have $\langle a, b \rangle \cong \langle d, \Psi(b)\rangle$ and hence $[d, \Psi(b)] \in \gamma_3(G)$ has order $\ell$. Continuing this way shows that $\gamma_i(G)$ contains an element of order $\ell$ for all $i \in \N$, which is a contradiction to $G$ being nilpotent. By symmetry, we obtain $\ord([a,b]) < \min\{\ord(a), \ord(b)\}$. In particular, we have $G' \subseteq \Omega_{s-1} = \mho$ (see Lemma~\ref{lemma:omegaagemo}). If $p$ is odd, then $G$ is powerful.
\end{proof}

From now on until the end of this section, we restrict our investigation to 2-tuple regular finite groups. 

\begin{lemma}\label{lemma:omegaonlypi}
	Let $G$ be a 2-tuple regular finite group of exponent $p^s$ for some $s \in \N$. For every $i \in \{0, \ldots, s\}$, every element in $\Omega_i$ has order at most $p^i$ and is a $p^{s-i}$-th power. In particular, $\Omega_i$ is 2-tuple regular. 
\end{lemma}

\begin{proof}
	Fix $i \in \{0, \ldots, s\}$. Let $\omega_1, \omega_2 \in \Omega_i$ be elements of order at most $p^i$, and assume that $\omega \coloneqq \omega_1 \omega_2$ has order $p^k$ for some $k > i$. Let $\omega' \in G$ with $\ord(\omega') = p^k$. Using 2-tuple regularity, we find $\omega_1', \omega_2' \in G$ such that the assignment $\omega \mapsto \omega'$, $\omega_1 \mapsto \omega_1'$, $\omega_2 \mapsto \omega_2'$ defines an isomorphism between $\langle \omega, \omega_1, \omega_2\rangle$ and $\langle \omega', \omega_1', \omega_2'\rangle$.  
	In particular, we have $\omega_1', \omega_2' \in \Omega_i$ and hence $\omega' = \omega_1' \omega_2' \in \Omega_i$. This yields $\Omega_k \subseteq \Omega_i$, which is a contradiction to $\Omega_i \subseteq \mho(\Omega_k) \subseteq \Phi(\Omega_k)$ (see Lemma~\ref{lemma:omegaagemo}). Hence $\Omega_i$ has exponent $p^i$. By Lemma~\ref{lemma:orderinn}, $\Omega_i$ is 2-tuple regular. Since all maximal cyclic subgroups of $G$ have order $p^s$ by Remark~\ref{rem:onetupleregular}, every element in $\Omega_i$ is a $p^{s-i}$-th power. We remark that for odd $p$, the statement of the lemma also follows directly by using $\Omega_i = \mho^{s-i}$ and the fact that~$G$ is powerful (see Lemmas~\ref{lemma:omegaagemo} and~\ref{lemma:1trpowerful}).
\end{proof}

\begin{lemma}\label{lemma:commutatorsomega}
	Let $G$ be a 2-tuple regular finite group of exponent $p^s$ for some $s \in \N$. For all $i,j \in \{1, \dots, s\}$, we have $[\Omega_i, \Omega_j] \in \ug$.
\end{lemma}

\begin{proof}
	Let $x \in [\Omega_i, \Omega_j]$. First assume $x = [a,b]$ with $a \in \Omega_i$ and $b \in \Omega_j$, and consider $y \in G$ with $\ord(y) = \ord(x)$. By 2-tuple regularity, there exists $a' \in G$ such that the assignment $x \mapsto y, a \mapsto a'$ defines an isomorphism between $\langle x,a \rangle$ and $\langle y,a'\rangle$. Let $\Psi \colon G \to G$ be an extending bijection. Then $[a,b] = x$ implies $[a', \Psi(b)] = y$, and hence $y \in [\Omega_i, \Omega_j]$ follows. Now let $x \in [\Omega_i, \Omega_j]$ be an arbitrary element and write $x = x_1 \cdots x_n$ for commutators $x_1, \dots, x_n$ as above. Since, for $k \in \{1, \dots, s\}$, the group~$\Omega_k$ consists precisely of the elements of order dividing $p^k$ in $G$ (see Lemma~\ref{lemma:omegaonlypi}), we have $\ord(x) \leq \max\{\ord(x_1), \dots, \ord(x_n)\}$. In particular, $[\Omega_i, \Omega_j]$ contains a commutator $[\omega_i, \omega_j]$ with $\omega_i \in \Omega_i$ and $\omega_j \in \Omega_j$ of order at least $\ord(x)$, so the claim follows by the first part of this proof.
\end{proof}
%

\begin{Remark}\label{rem:numbergenerators} \label{rem:powers}
Let $G$ be a 2-tuple regular finite group of exponent $p^s$ for some $s \in \N$. For every $i \in \{1, \dots, s\}$, we have $\Phi(\Omega_i) = \Omega_{i-1}$ (see Lemmas~\ref{lemma:omegaagemo}, \ref{lemma:1trpowerful} and~\ref{lemma:omegaonlypi}). In particular, we have $p^{d(\Omega_i)} = |\Omega_i : \Omega_{i-1}|$. Taking $p$-th powers defines a surjective map 
$\varphi \colon \Omega_i \to \Omega_{i-1}$. For $x,y \in \Omega_i$ with $x \Omega_{i-1} = y \Omega_{i-1}$, we have $x = y \omega$ for some $\omega \in \Omega_{i-1}$. Writing $C$ for the conjugacy class of $[y,\omega]$ in $G$, we have
$x^p = (y\omega)^p \equiv y^p \omega^p \pmod{\langle C \rangle}.$ Using $\omega^p \in \Omega_{i-2}$ as well as the fact that all elements in $\langle C\rangle$ have order less than $p^{i-1}$ by Lemma~\ref{lemma:1trpowerful}, we find $x^p \equiv y^p \pmod{\Omega_{i-2}}$. Hence~$\varphi$ induces a surjective map $\Omega_i/\Omega_{i-1} \to \Omega_{i-1}/\Omega_{i-2}$. In particular, we obtain 
\[|\Omega_s : \Omega_{s-1}| \geq |\Omega_{s-1}: \Omega_{s-2}| \geq \dots \geq |\Omega|.\] 
In particular, if there exist $i,j \in \{1,\dots, s\}$ with $i \geq j$ and $|\Omega_i : \Omega_{i-1}| = |\Omega_j : \Omega_{j-1}|$, taking $p^{i-j}$-th powers induces a bijection between $\Omega_i/\Omega_{i-1}$ and $\Omega_j/\Omega_{j-1}$.
\bigskip

By Remark~\ref{rem:onetupleregular}, we have $Z(G) = \Omega_j$ for some $j \in \{1, \dots, s\}$. If we additionally assume that $p$ is odd or that $\Omega_2 \subseteq Z(G)$ holds, then $d(Z(G)) \geq d(G) = d(\Omega_s)$ follows by \cite[Theorem III.12.2]{HUP67} and \cite[Theorem 1]{MAN71}, respectively. Moreover, as $Z(G)$ is a homocyclic group of exponent~$p^j$, we have $p^{d(Z(G))} = |\Omega_j : \Omega_{j-1}| = |\Omega|$. In this case, we therefore obtain \[|\Omega_s : \Omega_{s-1}| = |\Omega_{s-1}: \Omega_{s-2}| = \dots = |\Omega|.\] 
\end{Remark}

\subsection{Groups of prime power order which are $2$-tuple regular}\label{sec:classificationpgroups}
In this section, we classify the 2-tuple regular finite groups of prime power order. Throughout, let $p \in \BP$ and consider a 2-tuple regular finite group $G$ of exponent $p^s$ for some $s \in \N$. 
\medskip

Recall that for $i \in \{0, \ldots, s\}$, the group $\Omega_i$ consists precisely of the elements of order at most $p^i$ (see Lemma~\ref{lemma:omegaonlypi}). In other words, we have $\ug = \{1, \Omega, \ldots, \Omega_{s-1}, G\}$ and by Lemma~\ref{lemma:orderinn}, these groups are 2-tuple regular. For every $i \in \{0, \ldots, s\}$, there exists $j \in \{0, \ldots, s\}$ with $Z(\Omega_i) = \Omega_j$ (see Lemma~\ref{lemma:centralizerunionorderclasses}).

\begin{lemma}\label{lemma:technicallemma}
Let $G$ be a non-abelian $2$-tuple regular finite $p$-group. Let $i \in \{1, \ldots, s\}$ be minimal such that $\Omega_i$ is non-abelian, and let $j \in \{1, \ldots, i-1\}$ with $\Omega_j = Z(\Omega_i)$. If $|\Omega_{j+1} : \Omega_j| =  |\Omega|$ holds, we have $[a, \Omega_{j+1}] \subseteq \langle a^{p^{i-1}} \rangle$ for all $a \in \Omega_i$. 
\end{lemma}

\begin{proof}
Fix $a \in \Omega_i$, and consider the group homomorphism $\varphi \colon \Omega_{j+1} \to \Omega,\ \omega \mapsto [a,\omega]$. For a contradiction, assume $\Im(\varphi) \not \subseteq \langle a\rangle$. By Lemma~\ref{lemma:omegaagemo}, this implies $\ord(a) = p^i$.
Let $x,y \in \Omega \setminus \langle a \rangle$ with $x \in \Im(\varphi)$. The assignment $a \mapsto a$, $x \mapsto y$ defines an isomorphism between $\langle a,x \rangle$ and $\langle a,y \rangle$. By 2-tuple regularity, we obtain $y \in \Im(\varphi)$ and hence $\Omega \setminus \langle a \rangle \subseteq \Im(\varphi)$. Note that $\Omega$ is not cyclic as $G$ is neither cyclic nor a generalized quaternion group. Hence we obtain $\Omega = \langle \Omega \setminus \langle a \rangle \rangle \subseteq \Im(\varphi)$. We have $\langle a^{p^{i-j-1}}, \Omega_j \rangle \subseteq \Ker(\varphi)$ and $a^{p^{i-j-1}} \notin \Omega_j$ by Lemma~\ref{lemma:omegaonlypi}, which yields $|\Ker(\varphi)| \geq p \cdot |\Omega_j|$. Together with the assumption, this implies 
\[|\Omega| = |\Im(\varphi)| = |\Omega_{j+1}/\Ker(\varphi)| < |\Omega_{j+1}: \Omega_j| = |\Omega|,\] 
which is a contradiction. Hence $[a, \Omega_{j+1}] \subseteq \langle a \rangle$ follows for all $a \in \Omega_i$, which implies $[a, \Omega_{j+1}] \subseteq \langle a^{p^{i-1}} \rangle$. 
%
%
\end{proof}

We first consider the case that all quotients $\Omega_i/\Omega_{i-1}$ have the same order. By Remark~\ref{rem:numbergenerators}, this is always the case if $p$ is odd.

\begin{lemma}\label{lemma:indicesequalhomocyclic}
Let $G$ be a $2$-tuple regular finite group of exponent $p^s$ for some $s \in \N$. If 
\begin{equation}\label{eq:assindexequal}
|\Omega_s : \Omega_{s-1}| = |\Omega_{s-1} : \Omega_{s-2}| = \dots = |\Omega_2 : \Omega| = |\Omega|
\end{equation}
holds, then $G$ is homocyclic. 
\end{lemma}

\begin{proof}
By Remark~\ref{rem:onetupleregular}, it suffices to prove that $G$ is abelian. We show by induction that $\Omega_i$ is abelian for every $i \in \{1, \ldots, s\}$. For $i = 1$, this follows by Remark~\ref{rem:onetupleregular}. Now let $i \in \{2, \ldots, s\}$ and assume that $\Omega_{i-1}$ is abelian. By Lemma~\ref{lemma:commutatorsomega}, we have $\Omega_i' = \Omega_j$ for some $j \in \{0, \ldots, i-1\}$.
\medskip

First assume that $p$ is odd. Consider an element $x\in \Omega_i$ of order $p^i$, and suppose that $y \in \Omega_i$ satisfies $[x,y] \in \Omega \setminus \{1\}$. We choose $y$ of minimal order with this property. By the choice of $y$ and 2-tuple regularity, all elements of smaller order are central in $\Omega_i$. By Remark~\ref{rem:onetupleregular}, we have $y = z^q$ for some element $z \in \Omega_i$ of order $p^i$ and a suitable power $q$ of $p$. First assume $\langle x^{p^{i-1}} \rangle = \langle z^{p^{i-1}} \rangle$. By Remark~\ref{rem:powers}, taking $p^{i-1}$-th powers induces a bijection between $\Omega_i/\Omega_{i-1}$ and $\Omega$. Hence we have $z = x^r \omega$ for some $r \in \N$ not divisible by $p$ and some $\omega \in \Omega_{i-1}$. By Remark~\ref{rem:powers}, we find $y = z^q = \omega' x^{rq}$ with $\ord(\omega') < \ord(y)$, and $[x,y] = [x,z^q] = [x,\omega']$ holds, which is a contradiction to the minimality of $\ord(y)$. Hence $\langle x^{p^{i-1}} \rangle \cap \langle z^{p^{i-1}} \rangle = 1$ follows. 
\medskip

Assume $q > 1$. Applying the commutator relations together with the hypothesis that $\Omega_{i-1}$ is abelian, we obtain
\begin{equation}\label{eq:comm1}
[x,z^p] = [x,z]^p \prod_{t = 1}^{p-1} [z^t,[x,z]].
\end{equation}

Moreover, we obtain
\[ \prod_{t = 1}^{p-1} \left[z^t,[x,z]\right] \equiv \left[\left(\prod_{i= 1}^{p-1} z^t \right),[x,z]\right] \equiv \left[z^{\frac{p(p-1)}{2}},[x,z]\right] \pmod{[\Omega_i,[\Omega_i,\Omega_j]]}.
\]
Note that $[\Omega_i,[\Omega_i, \Omega_j]] \subseteq \Omega_{j-2}$ holds. Since $p$ is odd, we have $z^{p(p-1)/2} \in \Omega_{i-1}$. Since $\Omega_{i-1}$ is abelian, this implies $[z^{p(p-1)/2},[x,z]] = 1$.
By \eqref{eq:comm1}, we obtain $[x, z^p] \equiv [x, z]^p \pmod{\Omega_{j-2}}$. Similarly, one shows $[x^p,z] \equiv [x,z]^p \pmod{\Omega_{j-2}}$.  
Since $\Omega_{i-1}$ is abelian, we find $[x,z^q] = [x, z^p]^{q/p}$ and $[x^q, z] = [x^p,z]^{q/p}$. Together with the above and using $[x,z^q] \in \Omega$, this yields $[x,z^q] = [x^q,z]$. By Lemma~\ref{lemma:technicallemma}, this leads to the contradiction \[[x^q, z] = [x,z^q] \in \langle x^{p^{i-1}} \rangle \cap \langle z^{p^{i-1}} \rangle = 1.\] For $q = 1$, 
we obtain $[x,y] = [x,z] \in \langle x^{p^{i-1}} \rangle \cap \langle z^{p^{i-1}} \rangle = 1$ by Lemma~\ref{lemma:technicallemma}, which is a contradiction. 
%
\medskip

Now let $p = 2$ and fix $x \in \Omega_i$ with $\ord(x) = 2^i$. By Lemma~\ref{lemma:centralizerunionorderclasses}, we have $Z(\Omega_i) = \Omega_k$ for some $k \in \{1, \ldots, i\}$. Suppose $k < i$. Lemma~\ref{lemma:technicallemma} yields $[x,\Omega_{k+1}] \subseteq \langle x^{2^{i-1}} \rangle$. For $\omega_1, \omega_2 \in \Omega_{k+1}$ with $\omega_1^{2^k} = \omega_2^{2^k}$, we have $\omega_1 = \omega_2 \omega$ for some $\omega \in \Omega_k$ (see Remark~\ref{rem:powers}) and hence $[x,\omega_1] = [x,\omega_2]$ follows. Let $y_1 \in \Omega_{k+1}$ with $y_1^{2^k} \notin \langle x \rangle$. We show $[x,y_1] = x^{2^{i-1}}$. To this end, let $y_2 \in \Omega_{k+1}$ with $y_2^{2^k} \notin \langle x \rangle$. The assignment $x \mapsto x, y_1^{2^k}  \mapsto y_2^{2^k}$ defines an isomorphism between $\langle x, y_1^{2^k}\rangle$ and $\langle x, y_2^{2^k}\rangle$. For an extending bijection $\Psi \colon G \to G$, we have $[x, y_1] = [x, \Psi(y_1)]$ and $\Psi(y_1)^{2^k} = y_2^{2^k}$. By the above, we obtain $[x,y_1] = [x, \Psi(y_1)] = [x, y_2]$. Hence if $[x, y_1] = 1$ holds, then $[x,y_2] = 1$ follows. This implies $[x, \Omega_{k+1}] = 1$, which is a contradiction. Hence $[x, y_1] = x^{2^{i-1}}$ holds. Moreover, we obtain $[x,y] = 1$ for every element $y \in \Omega_{k+1}$ of order $2^{k+1}$ with $y^{2^k} \in \langle x \rangle$.
\medskip

For $|\Omega| > 4$, computing $[x,y_1y_2]$ for $y_1, y_2 \in \Omega_{k+1}$ with $\langle x^{2^{i-1}}, y_1^{2^k}, y_2^{2^k} \rangle \cong C_2^3$ leads to a contradiction. Now assume $|\Omega_i : \Omega_{i-1}| = |\Omega| = 4$ and write $\Omega_i/\Omega_{i-1} = \langle a \Omega_{i-1}, b \Omega_{i-1} \rangle$ for some $a,b \in \Omega_i$. By Lemma~\ref{lemma:commutatorsomega}, we find $\ell \in \N$ with $[\Omega_i, \Omega_{i-1}] = \Omega_\ell$. Then $\langle [a,b] [\Omega_i, \Omega_{i-1}] \rangle= \Omega_i'/ [\Omega_i,\Omega_{i-1}] = \Omega_j /\Omega_\ell$ is cyclic. If $\ell < j$ holds, $\Omega_j$ is cyclic, which is a contradiction. Now assume $\ell = j$. If $j = 1$ holds, then for any $\omega \in \Omega_i$, we have $[x, \omega^2] = [x,\omega]^2 = 1$. This yields $\Omega_j = \Omega_\ell = 1$, which is a contradiction. Hence $\Omega_2 \subseteq \Omega_j$ follows. For $\omega \in \Omega_{i-1}$, we have $[x^2,\omega] = 1$ as $\Omega_{i-1}$ is abelian. This yields $x [x,\omega]  x^{-1} = [x,\omega]^{-1}$, so conjugation with $x$ inverts $[x,\omega]$. Taking suitable powers, we thus find an element $y \in \Omega_2$ which is inverted by $x$. In particular, we have $k = 1$ and $[x,y] = y^2$. By the above, this yields $y^2 = x^{2^{i-1}}$. But then the above commutator relations yield $[x,y] = 1$, which is a contradiction.
If $|\Omega| = 2$ holds, $G$ is cyclic as it is not a generalized quaternion group by \eqref{eq:assindexequal}. Again, this is a contradiction. 
%
\end{proof}

\begin{lemma}\label{lemma:2groupsexceptions}
Let $G$ be a 2-tuple regular finite $2$-group. Either $G$ is homocyclic, or it is isomorphic to one of $Q_8$ and $G_{64}$.
\end{lemma}

\begin{proof}
Write $2^s$ ($s \in \N$) for the exponent of $G$. By Remark~\ref{rem:numbergenerators}, we have $|\Omega_s : \Omega_{s-1}| \geq \dots \geq |\Omega_2 : \Omega| \geq |\Omega|$. If equality holds in every step, $G$ is homocyclic by Lemma~\ref{lemma:indicesequalhomocyclic}. Assume that this is not the case and let $i \in \{2, \ldots, s\}$ be minimal with $|\Omega_i : \Omega_{i-1}| > |\Omega_{i-1}: \Omega_{i-2}|$. Then $\Omega_i$ is non-abelian as it would be homocyclic otherwise. Moreover, we have $\Omega_r(\Omega_{i-1}) = \Omega_{r}$ for $r \leq i-1$ by Lemmas~\ref{lemma:1trpowerful} and~\ref{lemma:omegaonlypi} and hence $\Omega_{i-1}$ is abelian by Lemma~\ref{lemma:indicesequalhomocyclic}. By Remark~\ref{rem:numbergenerators}, we have $\Omega_2 \not \subseteq Z(\Omega_i)$. Remark~\ref{rem:onetupleregular} then yields $Z(\Omega_i) = \Omega$. 
\medskip

For a contradiction, assume $i > 2$. For every element $x \in G$ of order $2^i$, we have $[x, \Omega_2] \subseteq \langle x^{2^{i-1}} \rangle$ (see Lemma~\ref{lemma:technicallemma}). Due to $i > 2$, we have $|\Omega_2 : \Omega| = |\Omega|$. Squaring therefore induces a bijection between $\Omega_2/\Omega$ and~$\Omega$. With this, one can show as in the proof of Lemma~\ref{lemma:indicesequalhomocyclic} that for $y \in \Omega_2$, we have $[x,y] = x^{2^{i-1}}$ if $y^2 \notin \langle x \rangle$ holds, and $[x,y] = 1$ otherwise. 
Now let $x,z \in \Omega_i$ be elements of order $2^i$ with $x^2 = z^2$ and set $\omega = x^{-1}z$. For $y \in \Omega_2$, we have $[y,\omega] = [y,x^{-1}z] = [y,x^{-1}] [y,z]$. 
With $\langle x^{2^{i-1}} \rangle = \langle z^{2^{i-1}} \rangle$, we obtain $[y,\omega] = 1$ for all $y \in \Omega_2$ and hence $\omega \in C_{\Omega_i}(\Omega_2) \subseteq \Omega_{i-1}$. Squaring thus induces a bijection between  $\Omega_i/\Omega_{i-1}$ and $\Omega_{i-1}/\Omega_{i-2}$, which is a contradiction.  
Hence $|\Omega_2 : \Omega| > |\Omega|$ follows and $\Omega_2$ is non-abelian. For $g \in G$ and $\omega \in \Omega_2$, we have $[g,\omega^2] = 1$, which implies $[G, \Omega_2] = \Omega$. Now assume that $G$ contains an element of order 8. By Remark~\ref{rem:onetupleregular}, every element in $\Omega_2$ is of the form $g^2$ for a suitable element $g \in G$. For any $\omega \in \Omega_2$, we then obtain $[\omega,g^2] = [\omega,g]^2 = 1$. Thus $\Omega_2$ is abelian, which is a contradiction. Hence $G = \Omega_2$ follows. 

\medskip

The remaining proof is an adaptation of the proof of \cite[Lemma 17]{CHE91} to our setting. Let $\{x_1, \ldots, x_n\}$ and $\{y_1, \ldots, y_m\}$ be $\F_2$-bases of $G/\Omega$ and $\Omega$, respectively. Consider the map $\sigma \colon G/\Omega \to \Omega$ induced by squaring and write $\sigma(\sum_{i = 1}^n \lambda_i x_i) = \sum_{j = 1}^m p_j(\lambda_1, \ldots, \lambda_n) y_j$ with quadratic forms $p_j \colon \F_2^n \to \F_2$ for $j \in \{1, \ldots, m\}$. By the Chevalley-Warning theorem, $\sigma$ has a nontrivial zero if $n > 2m$ holds, so $n \leq 2m$ follows. By 2-tuple regularity, the number of elements in $G$ squaring to a fixed element $z$ is constant for every choice of $z \in \Omega \setminus \{1\}$. Thus $2^m-1$ divides $2^{n}-1$. Due to $|G:\Omega| > |\Omega|$, we obtain $n = 2m$.
\medskip

For $m = 1$, we obtain $G \cong Q_8$. Now let $m > 1$. Let $a \in \Omega \setminus \{1\}$ and set $X = \sigma^{-1}(a)$. Furthermore, let $Y$ denote the preimage of $X$ in $G$. If $Y' \subseteq \langle a \rangle$ holds, then $H \coloneqq \{h \in G \colon h^2 \in \langle a \rangle \}$ is a subgroup of~$G$. But $H/\Omega$ has order $|X| +1 = 2^m+2$, which is a contradiction. Hence there exist $b \in \Omega \setminus \langle a\rangle$ and $y_1, y_2 \in Y$ with $[y_1, y_2] = b$. Let $b' \in \Omega \setminus \langle a \rangle$ and consider the assignment $y_1 \mapsto y_1$, $b \mapsto b'$. Due to $y_1^2 = a $ and $[y_1, b] = [y_1, b'] = 1$, it defines an isomorphism between $\langle y_1, b \rangle$ and $\langle y_1, b' \rangle$. Let $\Psi \colon G \to G$ be an extending bijection. Then $[y_1, \Psi(y_2)] = b'$ follows. Note that we have $\Psi(y_2) \in Y$ due to $y_1^2 = y_2^2 = a$. This argument shows that $|\{y \in Y \colon [y_1, y] = b\}|$ is constant for every choice of $b \in \Omega \setminus \langle a \rangle$. Due to $\Omega = Z(G)$, also $r \coloneqq |\{x \in X \colon [y_1,y]= b \text{ for some preimage $y \in Y$ of $x$}\}|$ is independent of the choice of $b \in \Omega \setminus \langle a \rangle$. In particular, this implies $r \cdot (|\Omega|- 2) \leq |X|$ and hence $r \leq \frac{2^m+1}{2^m -2} = 1 + \frac{3}{2^m-2}$ follows. Let $T$ be the graph with vertex set $X$ and edges between $x_1$ and~$x_2$ if $[y_1, y_2] = b$ for preimages $y_1$ and $y_2$ of $x_1$ and $x_2$, respectively. Then $T$ is a regular graph with $2^m+1$ vertices, so $r$ is even. This implies $m = 2$ and hence $|G| = 64$. Then $G \cong G_{64}$ can be deduced by either using GAP or arguing via a presentation of $G$ (see \cite[Lemma~17]{CHE91}). 
\end{proof}

We summarize our results in the following classification:

\begin{theorem}\label{theo:classificationpgroups}
Let $G$ be a $2$-tuple regular finite $p$-group. Then $G$ is homocyclic, or isomorphic to $Q_8$ or to~$G_{64}$. 
\end{theorem}

\begin{proof}
If $p$ is odd, $G$ is homocyclic by Remark~\ref{rem:numbergenerators} and Lemma~\ref{lemma:indicesequalhomocyclic}. Now let $p = 2$. If $G$ is not homocyclic, then $G$ is isomorphic to $Q_8$ or $G_{64}$ by Lemma~\ref{lemma:2groupsexceptions}. 
\end{proof}

\section{Classification of solvable $2$-tuple regular groups}\label{sec:solvable}
In this section, we classify the solvable 2-tuple regular finite groups. Our derivation roughly follows the steps in \cite{CHE91}. However, several fundamental principles used in their proofs do not hold when ultrahomogeneity is replaced by 2-tuple regularity (for instance, the statement of \cite[Lemma 20]{CHE91}), so we take alternative approaches. In Section~\ref{sec:semidirectproducts}, we derive preliminary results on certain semidirect products of $p$-groups. These form the basis for the classification of the solvable 2-tuple regular finite groups, which is given in Section~\ref{sec:classificationsolvable}.  

\subsection{Semidirect products}\label{sec:semidirectproducts}

In this section, we study 2-tuple regularity for finite groups of the form $M \rtimes Q$, where $Q$ is a group of prime power order with $\gcd(|M|,|Q|) = 1$. These results will form a central ingredient in the classification given in Section~\ref{sec:classificationsolvable}.
%

\begin{Remark}\label{rem:powerpnilpotent}
Let $G$ be a 2-tuple regular finite group of the form $G = M \rtimes Q$ with $Q \in \Syl_p(G)$ for some prime number $p \in \BP$.
\begin{enumerate}[(i)]
\item Let $q_1, q_2 \in Q$ be elements of the same order, and suppose that $q_1 = a_1^{p^n}$ holds for some $a_1 \in Q$ and $n \in \N$. By Remark~\ref{rem:centralizer}, there exists an element $b \in G$ of order $\ord(a_1)$ with $q_2 = b^{p^n}$. Now we have $b = m a_2 m^{-1}$ for some $m \in M$ and $a_2 \in Q$, and hence $q_2 = m a_2^{p^n} m^{-1}$ follows. Since $G$ is $p$-nilpotent, $q_2$ and $a_2^{p^n}$ are already conjugate in~$Q$. Writing $q_2 = s a_2^{p^n} s^{-1}$ for some $s \in Q$, we obtain $q_2 = (sa_2s^{-1})^{p^n}$. In other words, also $q_2$ is a $p^n$-th power of an element of order $\ord(a_1)$ in $Q$.
\item Let $R \in \mathcal{U}(Q)$ be a subgroup of $Q$ of exponent $p^\ell$ for some $\ell \in \N$. By (i), every element in $R$ is a power of an element of order $p^\ell$ in $Q$. In particular, these elements generate $R$.
Moreover, for $S \in \mathcal{U}(Q)$ with $S \subsetneq R$, this implies $S \subseteq \Phi(R)$.
\end{enumerate}
\end{Remark}

We first consider the following frequently occurring special case:

\begin{lemma}\label{lemma:poweringgroups}
Let $G$ be a 2-tuple regular finite group. Suppose $G = (P_1 \times \dots \times P_s) \rtimes Q$ for Sylow subgroups $P_1, \ldots, P_s, Q$ of pairwise coprime order. If, for every $i \in \{1, \ldots, s\}$, the group $P_i$ is abelian and~$Q$ induces a group of power automorphisms on $\Omega(P_i)$, then one of the following holds: 
\begin{enumerate}[(i)]
	\item $G \cong P_1 \times \dots \times P_n \times Q$. 
	\item $Q \cong C_{2^n}$ for some $n \in \N$, and for every $i \in \{1, \ldots, n\}$, the group $Q$ either centralizes or inverts $P_i$.
\end{enumerate}
\end{lemma}

\begin{proof}
Let $i \in \{1, \ldots, s\}$ and write $|P_i| = p_i^{\ell_i}$ for some $p_i \in \BP$ and $\ell_i \in \N$. Fix $q \in Q$ and let $\lambda_q \in \Z$ with $q a q^{-1} = a^{\lambda_q}$ for all $a \in \Omega(P_i)$. Let $\lambda' \in \Z$ with $\lambda_q \lambda' \equiv 1 \pmod{p_i}$. We then have $q^{-1} a q = a^{\lambda'}$ for all $a \in \Omega(P_i)$. Let $\Psi \colon G \to G$ be an extending bijection extending the assignment $q \mapsto q^{-1}$ and let $\varphi \colon \langle q, a\rangle \to \langle q^{-1}, \Psi(a) \rangle$ denote the isomorphism defined by $q \mapsto q^{-1},\ a \mapsto \Psi(a)$. We have 
\[\varphi(qaq^{-1}) = \varphi(q) \varphi(a) \varphi(q)^{-1}= q^{-1} \Psi(a)q = \Psi(a)^{\lambda'}.\] 
On the other hand, we obtain $\varphi(qaq^{-1}) = \varphi(a^{\lambda_q}) = \varphi(a)^{\lambda_q} = \Psi(a)^{\lambda_q}$. Comparing the two expressions yields $\lambda_q \equiv \lambda' \pmod{p_i}$, which implies $\lambda_q^2 \equiv 1 \pmod{p_i}$. In particular, we obtain $|Q/C_Q(\Omega(P_i))| \leq 2$. If $|Q|$ is odd, $Q$ centralizes $\Omega(P_i)$, and by \cite[Theorem~5.2.4]{GOR68}, $Q$ acts trivially on $P_i$. This implies $G \cong P_1 \times \dots \times P_s \times Q$. If $Q$ is a Sylow 2-subgroup of $G$, it centralizes or inverts $\Omega(P_i)$, and hence $P_i$ again by \cite[Theorem~5.2.4]{GOR68}. Assume that $Q$ acts on $P_i$ by inversion. Due to $C_G(\Omega(P_i))  \in \ug$ (see Lemmas~\ref{lemma:centralizerunionorderclasses} and~\ref{lemma:oqfo}), we find $C_Q(\Omega(P_i)) \in \mathcal{U}(Q)$. By Remark~\ref{rem:powerpnilpotent}, this implies $C_Q(\Omega(P_i)) \subseteq \Phi(Q)$, and hence $Q$ is cyclic as $Q/C_Q(\Omega(P_i))$ is cyclic.
\end{proof}

This special case is used in the following more general situation:

\begin{lemma}\label{lemma:2groupacting}
Let $G$ be a 2-tuple regular finite group of the form $G = \fo \rtimes Q$ with $Q \in \Syl_2(G)$. Then one of the following cases applies: 
\begin{enumerate}[(i)]
\item $G \cong \fo\times Q$. 
\item $Q \cong C_{2^n}$, and for every odd $p \in \BP$, the group $Q$ either centralizes or inverts $\op$. 
\item $Q \cong Q_8$ centralizes $\op$ for every $p \in \BP_{\geq 5}$ and acts on $O_3(G) \cong C_3^2$ as $\SL(2,3)'$.
\end{enumerate}
\end{lemma}

\begin{proof}
By Lemma~\ref{lemma:oqfo} and Theorem~\ref{theo:classificationpgroups}, the group $\fo$ is abelian. We assume that $Q$ acts nontrivially on $\fo$ since we are in case (i) otherwise. If, for every odd prime $p \in \BP$, the group $Q$ induces a group of power automorphisms on $\Omega(\op)$, we are in case (ii) by Lemma~\ref{lemma:poweringgroups}. Otherwise there exists an odd prime $p \in \BP$ and elements $x,y \in \Omega(\op)$ with $\langle x \rangle \cap \langle y \rangle = 1$ that are conjugate in~$G$. In particular, $R \coloneqq \Omega(\op)$ is not cyclic. By Lemma~\ref{lemma:commutators2tupleregular}, all nontrivial elements in~$R$ are conjugate in $G$. Write $C_G(R) = F_0(G) \tilde{Q}$ for some proper subgroup $\tilde{Q}$ of $Q$. 
By Lemmas~\ref{lemma:centralizerunionorderclasses} and~\ref{lemma:oqfo}, we have $C_G(R) \in \ug$,
which implies $\tilde{Q} \in \mathcal{U}(Q)$. By Remark~\ref{rem:powerpnilpotent}, we have $\tilde{Q} \subseteq \Phi(Q)$. 
\medskip

The 2-group $\bar{G} \coloneqq G/C_G(R) \cong Q/\tilde{Q}$ has a cyclic center by \cite[Theorem 3.2.2]{GOR68}. Consider an element $x \in G$ with $x C_G(R) \in Z(\bar{G})$, so $[x,G] \subseteq C_G(R)$. For $y \in G$ with $\ord(y) = \ord(x)$, we have $[y,G] \subseteq C_G(R)$ due to 2-tuple regularity and $C_G(R) \in \ug$. This implies $y C_G(R)\in Z(\bar{G})$, and hence the preimage of $Z(\bar{G})$ in~$G$ is contained in $\ug$. 
%
%
%
By Lemma~\ref{lemma:preimageorder}, also $Z(\bar{G}) \in \mathcal{U}(\bar{G})$ holds and hence $Z(\bar{G}) = \Omega_j(\bar{G})$ follows for some $j \in \N$. Since $Z(\bar{G})$ is a nontrivial cyclic 2-group, this yields $|\Omega(\bar{G})| = 2$. Let $L$ be the preimage of $\Omega(Q/\tilde{Q})$ in~$Q$. Then we have $|L :\tilde{Q}| = 2$. Due to $\tilde{Q} \in \mathcal{U}(Q)$, we have $L \in \mathcal{U}(Q)$. By Remark~\ref{rem:powerpnilpotent}, we obtain $\tilde{Q} \subseteq \Phi(L)$. 
By Burnside's theorem, $L$, and hence $\tilde{Q}$, is cyclic. 
If $\tilde{Q}$ is nontrivial, then $\tilde{Q}$, and hence~$Q$, contains a single involution. If $\tilde{Q} = 1$, then $Z(Q)$ is cyclic and $Q$ contains a single involution. In either case,~$Q$ is cyclic or a generalized quaternion group. 

\medskip

\medskip

First assume that $Q$ is cyclic. Since all nontrivial elements in $R$ are conjugate by elements of $Q$, we obtain that $|R|-1$ is a power of $2$. By Mih\u{a}ilescu's Theorem (see \cite{MIH04}), this implies $|R| = 9$ or $|R|  \in \BP$. Since~$R$ is not cyclic, we obtain $R \cong C_3^2$ and hence $Q/\tilde{Q} \leq \GL(2,3)$. Since $Q$ is cyclic and all nontrivial elements in $R$ are conjugate in $G$, we obtain $Q/\tilde{Q} \cong C_8$. In particular, this group acts regularly on $R \setminus \{1\}$. Hence there exist pairs $(x,y) \in  R \times R$ with $\langle x,y \rangle = R$ for which $x$ and $y$ are conjugate by an element of order $|Q|$, and others that do not have this property. This is a contradiction to Lemma~\ref{lemma:commutators2tupleregular}.
 
\medskip

Now assume that~$Q$ is a generalized quaternion group. By Remark~\ref{rem:centralizer} and Example~\ref{ex:quaterniongroup}, this implies $Q \cong Q_8$. Since $Q/\tilde{Q}$ has a cyclic center, we obtain $\tilde{Q} = 1$, so $Q$ acts faithfully on~$R$ and transitively on $R \setminus \{1\}$. This implies $R \cong C_3 \times C_3$ and $Q$ is identified with the derived subgroup of $\SL(2,3)$. For $p \in \BP_{\geq 5}$, this argument shows that $Q$ induces a group of power automorphisms on $\op$. As in Lemma~\ref{lemma:poweringgroups}, it follows that~$Q$ centralizes $\op$. Hence $G$ has the structure described in~(iii). 
\end{proof}

We now study the case that $G$ has a normal Sylow 2-subgroup, beginning with the following special case:

\begin{Example}\label{ex:affinelineargroup}
Consider the affine linear group $G \coloneqq \AGL(1, 2^d)$ for $d \in \N_{\geq 2}$ and suppose that $G$ is 2-tuple regular. 
We have $G = G' \rtimes \langle h \rangle$ for some $h \in G$ of order $2^d -1$. Let $1  \neq a \in G'$ and fix $k \in \{1, \ldots, 2^d-2\}$. Since $G'$ is an elementary abelian $2$-group, the assignment $a \mapsto a, hah^{-1} \mapsto h^k a h^{-k}$ defines an isomorphism between $\langle a, hah^{-1}\rangle$ and $\langle a, h^ka h^{-k}\rangle$. Let $\Psi \colon G \to G$ be an extending bijection and write $\Psi(h) = h^\ell a'$ for some $\ell \in \Z$ and $a' \in G'$. Let $\varphi$ denote the automorphism of $G$ defined by $a \mapsto a, hah^{-1} \mapsto h^k a h^{-k}, h \mapsto \Psi(h)$. Then 
\[
h^k a h^{-k} = \varphi(hah^{-1}) = \varphi(h) \varphi(a) \varphi(h)^{-1} = (h^\ell a') a (h^\ell a')^{-1} = h^\ell a h^{-\ell}.\] 
This forces $h^\ell = h^k$. In particular, $\Aut(G)$ acts transitively on the nontrivial cosets of $G'$ in $G$. This implies $d = 2$ and hence $G = \AGL(1,4) \cong A_4$ follows.
\end{Example}

\begin{lemma}\label{lemma:o2gpodd}
	Let $G$ be a 2-tuple regular finite group of the form $G = O_2(G) \rtimes P$ with $P \in \Syl_p(G)$ for some $p \in \BP$. 
	Then $P$ centralizes $O_2(G)$, or $G$ is isomorphic to one of $A_4$, $\SL(2,3)$, and $G_{192}$. 
\end{lemma}

\begin{proof}
	Assume that $P$ acts nontrivially on $O_2(G)$. 
	By Lemmas~\ref{lemma:centralizerunionorderclasses} and~\ref{lemma:oqfo}, we have $C_G(O_2(G)) \in \ug$. Let $\bar{P} \coloneqq P/C_P(O_2(G))$. As in the proof of Lemma~\ref{lemma:2groupacting}, one can show $C_P(O_2(G)) \in \mathcal{U}(P)$ as well as $Z(\bar{P}) \in \mathcal{U}(\bar{P})$. 
	By Lemma~\ref{lemma:oqfo} and Theorem~\ref{theo:classificationpgroups}, the group $O_2(G)$ is either homocyclic or isomorphic to $Q_8$ or $G_{64}$. Moreover, $\bar{P}$ is isomorphic to a subgroup of $\Aut(O_2(G))$. 
	\medskip
	
	We show $|\bar{P}| = p$. To this end, first assume that $O_2(G)$ is homocyclic. If $O_2(G)$ is cyclic, then $\bar{P}$ is isomorphic to a subgroup of the 2-group $\Aut(O_2(G))$, which is a contradiction. Now suppose that $O_2(G)$ is not elementary abelian. It is then easy to see that all elements of order $4$ are conjugate in $G$. In particular, every element is conjugate to its inverse, which is a contradiction to $|P|$ being odd. Hence $O_2(G)$ is elementary abelian and~$\bar{P}$ permutes the nontrivial elements in $O_2(G)$ transitively. By \cite[Theorem 3.2.2]{GOR68}, the group $Z(\bar{P})$ is cyclic. Due to $Z(\bar{P}) \in \mathcal{U}(\bar{P})$, also $\Omega(\bar{P})$ and hence $\bar{P}$ are cyclic. This yields $|\bar{P}| = |O_2(G)| -1$, and hence $|\bar{P}| = p$ follows by Mih\u{a}ilescu's Theorem. For $O_2(G) \cong Q_8$, we obtain $\bar{P} \cong C_3$ due to $|\Aut(Q_8)| = 24$. For $O_2(G) \cong G_{64}$, we have $\bar{P} \cong C_3$ or $\bar{P} \cong C_5$ due to $|\Aut(G_{64})| = 15360$. In any case, we find $|\bar{P}| = p$. 
\medskip

Due to $C_P(O_2(G)) \in \mathcal{U}(P)$, we have $C_P(O_2(G)) \subseteq \Phi(P)$ (see Remark~\ref{rem:powerpnilpotent}). 
Due to $|P: C_P(O_2(G))| = p$, the group~$P$ is cyclic by Burnside's theorem. 
Write $P = \langle h \rangle$ for some $h \in G$ and suppose $|P| > p$ for a contradiction. 
Let $a$ be an element of maximal order in $O_2(G)$. Set $b \coloneqq h^{-1} a h$ and consider the assignment $a \mapsto a$, $h^pb \mapsto h^{-p} b$. Due to $h^p  \in Z(G)$, it defines an isomorphism between $\langle a, h^pb\rangle \cong \langle a,b\rangle \times \langle h^p \rangle$ and $\langle a, h^{-p}b \rangle \cong \langle a,b \rangle \times \langle h^{-p}\rangle$. Let $\Psi \colon G \to G$ be an extending bijection and write $\Psi(h) = h^\ell g $ for $g \in O_2(G)$ and $\ell \in \Z$. Let $\varphi$ denote the automorphism of $G$ defined by $a \mapsto a$, $h^p b \mapsto h^{-p} b$, $h \mapsto \Psi(h)$. We have 
\[\varphi(h^p) =  (h^\ell g)^p = h^{\ell p} g^p [g, h^{\ell(p-1)}] \cdots [g, h^{\ell}] = h^{\ell p} g^{p+1} = h^{\ell p} \pmod{O_2(G)'},\] since $p$ is odd. Let $c \in O_2(G)'$ with $\varphi(h^p) = h^{\ell p}c$. 
Then
\[\varphi(h^pb) = \varphi(h^p) \varphi(h^{-1} a h) =  h^{\ell p} c h^{-\ell} a h^{\ell}.\] 
On the other hand, we have $\varphi(h^pb) = h^{-p}b = h^{-p} h^{-1}a h$ by assumption. We thus obtain $h^{\ell p} O_2(G)= h^{-p} O_2(G)$, which yields $\ell  \equiv -1 \pmod{|P|/p}$. In particular, $h^{\ell}$ acts as $h^{-1}$ on $O_2(G)$. Thus $ch^{-\ell} ah^{\ell}  = ch a h^{-1}  \neq h^{-1}a h =  b$ as $\ord(c) < \ord(a)$ holds in all cases. This is a contradiction. Hence $|P| = p$ follows. If $O_2(G)$ is elementary abelian, this yields $G \cong \AGL(1,2^d)$. By Example~\ref{ex:affinelineargroup}, we then obtain $G \cong A_4$. If $O_2(G) \cong Q_8$, we obtain $G \cong \SL(2,3)$. Finally, for $O_2(G) \cong G_{64}$, we have $P \cong C_3$ or $P \cong C_5$. Using GAP, it is then easily verified that $P \cong C_3$ holds and that we have $G \cong G_{192}$. 
\end{proof}

\subsection{Classification}\label{sec:classificationsolvable}

Using the results of Section~\ref{sec:semidirectproducts}, we now classify the solvable 2-tuple regular finite groups. Throughout, we assume that $G$ is a 2-tuple regular finite group. 
\medskip

We first study the action of certain normal subgroups of $G$ on $F_0(G) = \prod_{p \in \mathbb{P} \setminus \{2\}} \op$. Recall that $F_0(G)$ is abelian by Lemma~\ref{lemma:oqfo} and Theorem~\ref{theo:classificationpgroups}. In particular, we have $\fo \subseteq C_G(F(G))$.

\begin{lemma}\label{lemma:oqmodf}
Let $q \in \BP$ be an odd prime. Then the action on $\fo$ induced by $O_q(G/F(G))$ is trivial.
\end{lemma}

\begin{proof}
Let $B$ be the preimage of $O_q(G/F(G))$ in $G$. By Lemma~\ref{lemma:oqfo}, the group $B$ is $2$-tuple regular. By the Schur-Zassenhaus Theorem, we have $B = (\prod_{p \in \BP \setminus \{q\}} O_p(G)) \rtimes Q$ with $Q \in \Syl_q(B)$.
Let $p \in \BP \setminus \{q\}$ be an odd prime and consider $1 \neq a \in \Omega(\op)$. If $xax^{-1} \notin \langle a \rangle$ for some $x \in B$, then $a$ and $a^{-1}$ are conjugate in~$B$ (see Lemma~\ref{lemma:commutators2tupleregular}). Write $a^{-1} = bab^{-1}$ for some $b \in B$. Due to $F(G) \subseteq C_G(a)$, we may assume $b \in Q$, which is a contradiction. Hence $x$ normalizes every subgroup of $\Omega(\op)$, so it acts on $\Omega(\op)$ by a power automorphism. 
By Lemma~\ref{lemma:poweringgroups}, $Q$ acts trivially on $\fo$.
\end{proof}

\begin{Corollary}\label{cor:fo}
We have $F_0(G/ \fo) = 1$. In particular, if $|G|$ is odd, then $G$ is abelian. 
\end{Corollary}

\begin{proof}
Let $q \in \BP$ be an odd prime and let $X$ be the preimage of $O_q(G/\fo)$ in $G$. Write \[X = \left(\prod_{p \in \BP \setminus \{2,q\}} O_p(G)\right) \rtimes Q\] with $Q \in \Syl_q(X)$ as in the proof of Lemma~\ref{lemma:oqmodf}. Since $X F(G)/F(G) \subseteq O_q(G/F(G))$ holds, $X$ acts trivially on $F_0(G)$ and hence we have $X = (\prod_{p \in \BP \setminus \{2,q\}} O_p(G)) \times Q \subseteq \fo$.
This implies $O_q(G/\fo) = 1$.
\end{proof}

%

%
%

In the following, let $H$ be the preimage of $F(G/\fo)$ in $G$. By Corollary~\ref{cor:fo}, we have $F(G/\fo) = O_2(G/\fo)$. In particular, this yields $H = \fo \rtimes S$ for some $S \in \Syl_2(H)$. By Lemma~\ref{lemma:oqfo}, we have $H \in \ug$ and $H$ is 2-tuple regular. In the following, we distinguish the cases $H \neq F(G)$ and $H = F(G)$.

\begin{lemma}\label{lemma:hneqfg}
Assume $H \neq F(G)$. Then one of the following holds: 
	\begin{enumerate}[(i)]
	\item We have $G \cong F_0(G) \rtimes Q$ with $Q \cong C_{2^n}$ for some $n \in \N$, and for every odd prime $p \in \BP$, the group $Q$ centralizes or inverts $\op$. 
	\item We have $G = A \times B$ for a 2-tuple regular group $A$ of order coprime to 6 and $B \cong C_3^2 \rtimes Q_8$ with $Q_8$ acting as $\SL(2,3)'$ on $C_3^2$.
	\end{enumerate}
\end{lemma}

\begin{proof}
Due to $H \neq F(G)$, $S$ acts nontrivially on $F_0(G)$. By Lemma~\ref{lemma:2groupacting}, either $S$ is cyclic, or we have 
\begin{equation}\label{eq:h}
	H \cong \left(\prod_{p \in \BP_{\geq 5}} \op\right) \times (C_3^2 \rtimes Q_8).
\end{equation}

Let $Q \in \Syl_2(G)$. In either of the above cases, $Q$ contains a single involution due to $H \in \ug$. Hence $Q$ is cyclic or a generalized quaternion group. In the first case, $G$ has a normal 2-complement, so we obtain $G = F_0(G) \rtimes Q$, and $Q$ centralizes or inverts $\op$ for every odd prime $p \in \BP$ (see Lemma~\ref{lemma:2groupacting}). 
\medskip

Now let $Q$ be a generalized quaternion group. By Remark~\ref{rem:centralizer} and Example~\ref{ex:quaterniongroup}, we obtain $Q \cong Q_8$. We argue analogously to \cite[Lemma~24]{CHE91} that $G$ has a normal 2-complement. By \eqref{eq:h}, we have $O_3(G) \cong C_3^2$. The group $C_G(O_3(G))$ is a 2-tuple regular group of odd order and hence abelian (see Corollary~\ref{cor:fo}), so we obtain $C_G(O_3(G)) = \fo$. Hence $G/\fo$ embeds into $\Aut(O_3(G)) \cong \GL(2,3)$. In particular, $G/\fo$ is a $\{2,3\}$-group. Hence $G = \prod_{p \in \BP_{\geq 5}} \op \times K$ for a $\{2,3\}$-subgroup $K$ of $G$ with $K/O_3(G) \leq \GL(2,3)$. 
\medskip

Assume that $K/O_3(G)$ is not a 2-group. Let $k \in K$ is an element whose image $k O_3(G)  \in K/O_3(G)$ has order~$3$, and let $x \in K$ act on $O_3(G)$ by inversion. Since $k^3 \in O_3(G) \setminus \{1\}$ holds, we have $[x,k^3] = k^{3} \neq 1$. On the other hand, we find $[x,k^3] = [x,k] \cdot k [x,k]k^{-1} \cdot k^2 [x,k] k^{-2} = 1$. This is a contradiction. Hence $K/O_3(G)$ is a 2-group, so $K \cong C_3^2 \rtimes Q_8$ follows. 
\end{proof}


\begin{lemma}\label{lemma:hfg}
	Assume $H = F(G)$. Then one of the following holds: 
	\begin{enumerate}[(i)]
\item $G$ is nilpotent. 
\item We have $G \cong A \times B$ for a 2-tuple regular group $A$ of order coprime to 6 and $B \in \{A_4,\SL(2,3), G_{192}\}$.
	\end{enumerate}
\end{lemma}

\begin{proof}
Suppose that $G$ is not nilpotent, and let $\bar{G} \coloneqq G/F(G)$. Since $G$ is solvable, we have $O_p(\bar{G}) \neq 1$ for some $p \in \BP$. 
Due to $H = F(G)$ and Corollary~\ref{cor:fo}, we have $O_2(G/\fo) = F(G/\fo) = F(G)/\fo$. Now 
\[O_2(\bar{G}) \cong O_2\bigl((G/\fo)\big/(F(G)/\fo)\bigr) = O_2\bigl((G/\fo)/O_2(G/\fo)\bigr) = 1.\] 
Hence $p \neq 2$ follows. By Lemma~\ref{lemma:oqmodf}, the action of $O_p(\bar{G})$ induced on $F_0(G)$ is trivial. Let $S$ denote the preimage of $O_p(\bar{G})$ in $G$.

\medskip

Again, $O_2(G)$ is either homocyclic or isomorphic to $Q_8$ or $G_{64}$. First assume that $O_2(G)$ is homocyclic, so $F(G)$ is abelian. In particular, $O_p(\bar{G})$ acts faithfully on $O_2(G)$. As in the proof of Lemma~\ref{lemma:o2gpodd}, one can show that this implies $O_2(G) \cong C_2^2$ and $O_p(\bar{G}) \cong C_3$. Now assume $O_2(G) \cong Q_8$. Since $S/C_S(O_2(G))$ is a subgroup of $\Aut(Q_8)$ which is a not a 2-group, we obtain $O_p(\bar{G}) \cong C_3$. If $O_2(G) \cong G_{64}$ holds, we analogously obtain $O_p(\bar{G}) \cong C_3$ or $O_p(\bar{G}) \cong C_5$. As in Lemma~\ref{lemma:o2gpodd}, the latter option can be excluded, so we find $O_p(\bar{G}) \cong C_3$ in all cases. In particular, we must have $p = 3$, so $F(\bar{G}) = O_3(\bar{G})$ follows. 
\medskip

For a contradiction, assume that $F(\bar{G})$ is a proper subgroup of $\bar{G}$. Due to $C_{\bar{G}}(F(\bar{G})) = F(\bar{G})$, we then have $\bar{G} \cong S_3$. 
%
We argue as in the proof of \cite[Lemma~26]{CHE91} that this situation leads to a contradiction. Assume $O_2(G) \cong C_2^2$ and let $x \in G$ be a 2-element such that $x F(G) \in \bar{G}$ is an involution. Then $\ord(x) = 4$ follows. Let $b \in O_2(G) \setminus \langle x^2 \rangle$. Then $xbx^{-1} = x^2 b$ yields $(xb)^2 = xbxb = xbx^{-1}x^2b = 1$, which is a contradiction. 
Now let $O_2(G) \cong Q_8$. Then $G$ induces the full automorphism group of $O_2(G)$. Let $g \in G$ be a 2-element inducing an outer automorphism of order~2 of $O_2(G)$. Then $g^2 \in C_G(O_2(G)) \subseteq F(G)$ follows. This implies $g^2 \in Z(O_2(G)),$ so $\ord(g) \leq 4$, and hence $g \in O_2(G)$ follows, which is a contradiction. The case $O_2(G) \cong G_{64}$ cannot occur as $\Aut(G_{64})$ does not contain a normal subgroup isomorphic to $S_3$. 
\medskip

Hence $\bar{G} = O_3(\bar{G}) \cong C_3$ follows. This yields $G = A \times B$ with $A \coloneqq \prod_{q \in \BP_{\geq 5}} O_q(G)$ and $B \coloneqq O_2(G) \rtimes P$ for $P \in \Syl_3(G)$. By Lemma~\ref{lemma:o2gpodd}, we obtain $B \in \{A_4, \SL(2,3), G_{192}\}$, which concludes the proof.
\end{proof}

In summary, we obtain the following characterization of solvable $2$-tuple regular groups: 

\begin{theorem}\label{theo:solvable}
Let $G$ be a finite solvable group. Then $G$ is 2-tuple regular if and only if $G = G_1 \times G_2$ for subgroups $G_1$ and $G_2$ of coprime order, where $G_1$ is an abelian group with homocyclic Sylow subgroups, and one of the following holds for $G_2$: 
\begin{enumerate}[(i)]
\item $G_2 = 1$.  
\item $G_2$ is isomorphic to $Q_8$ or $G_{64}$. 
\item $G_2 \cong (M_1 \times \dots \times M_s) \rtimes Q$, where $M_1, \ldots, M_s$ are homocyclic groups of odd, pairwise coprime prime power order, and for every $i \in \{1, \ldots, s\}$, the group $Q \cong C_{2^n}$ ($n \in \N$) acts on $M_i$ by inversion. 
\item $G_2$ is isomorphic to $A_4$, to $\SL(2,3)$, to $C_3^2 \rtimes Q_8$, or to $G_{192}$.
\end{enumerate} 
\end{theorem}

\begin{proof}
First assume that $G$ is 2-tuple regular. If $G$ is nilpotent, it is a direct product of its Sylow subgroups. By Theorem~\ref{theo:classificationpgroups}, these are homocyclic or isomorphic to $Q_8$ or $G_{64}$. Now assume that $G$ is not nilpotent. By Lemmas~\ref{lemma:hneqfg} and~\ref{lemma:hfg} together with Corollary~\ref{cor:fo}, we have $G = G_1 \times G_2$, where $G_1$ is an abelian 2-tuple regular group and $G_2$ is isomorphic to one of the groups in (iii) or (iv).
Conversely, note that the above groups are ultrahomogeneous by \cite{CHE91} and hence 2-tuple regular. 
\end{proof}

\section{General case}\label{sec:general}

In this section, we classify the 2-tuple regular finite groups. To this end, we consider the cases that the layer $E(G)$ of~$G$ is trivial (Section~\ref{sec:triviallayer}) or nontrivial (Section~\ref{sec:nontriviallayer}) individually. In Section~\ref{sec:mainproofs}, we summarize the results of this paper by proving Theorems~\ref{theo:a} and~\ref{theo:b}. 

\subsection{Trivial layer}\label{sec:triviallayer}

In this section, we study 2-tuple regular finite groups $G$ with $E(G) = 1$. In analogy to \cite[Proposition 1]{CHE00}, we prove that these groups are solvable. Again, the proof given therein is not transferable to the 2-tuple regular setting. Therefore, we take a different approach using the classification of the transitive linear groups:

\begin{theorem}[{\cite[Chapter 5]{HER74}}]\label{theo:hering}
Let $p \in \BP$ and consider an $\F_p$-vector space $V$ of dimension $n \in \N$. Let~$A$ be a subgroup of $\GL(V)$ that acts transitively on the set of nonzero vectors of $V$. Let $L$ be a subfield of $\operatorname{Hom}(V,V)$ containing the identity and being maximal with respect to $a L a^{-1} = L$ for all $ a \in A$. Set $p^m \coloneqq |L|$ and let $n^\ast \in \N$ with $n = n^\ast m$. We write $V_L$ for $V$ viewed as an $L$-vector space of dimension $n^\ast$. Then one of the following holds:  	

\begin{enumerate}[(a)]
	\item We have $\SL(V_L) \trianglelefteq A \subseteq \operatorname{\Gamma L}(V_L)$. 
	\item There exists a non-degenerate symplectic form on $V_L$, and the corresponding symplectic group $\operatorname{Sp}(V_L)$ is normal in $A$.
	\item We have $n^\ast = 6$, $p = 2$ and $A$ contains a normal subgroup isomorphic to the Chevalley group $G_2(2^m)$ for $m \geq 2$.
	\item The group $A$ contains a normal subgroup $E$ isomorphic to an extraspecial group of order $2^{n+1}$. Furthermore, we have $C_A(E) = Z(A)$ and $A/E Z(A)$ is faithfully represented on $E/Z(E)$. If $n = 2$, we have $n^\ast =  2$ and $|L| \in \{3,5,7,11,23\}$. If $n > 2$, then $n^\ast = n = 4$ and $|L| = 3$ holds.
	\item One of the following exceptional cases arises: 
	\begin{enumerate}[(E1)]
	\item $A^{(\infty)} \cong \SL(2,5)$, $n^\ast = 2$ and $|L| \in \{9,11,19,29,59\}$, where $A^{(\infty)}$ denotes the last term of the derived series of $A$.
	\item $A \cong A_6$, $n^\ast = 4$ and $|L| = 2$. 
	\item $A \cong A_7$, $n^\ast = 4$ and $|L| = 2$. 
	\item $A \cong \SL(2,13)$, $n^\ast = 6$ and $|L| = 3$. 
	\item $A \cong \PSU(3,3)$, $n^\ast = 6$ and $|L| = 2$. 
	\end{enumerate}

		\end{enumerate}
\end{theorem}

\begin{Remark}\label{rem:solvablequotient}
Let $G$ be a 2-tuple regular finite group with $E(G) = 1$. Then \cite[Theorem~6.5.8]{KUR98} yields $C_G(F(G)) = Z(F(G))$. Since this group is solvable, $G$ is solvable if and only if $G/C_G(F(G))$ is solvable. In order to show the latter, it suffices to show that that $G/C_G(\op)$ is solvable for every $p \in \BP$. The claim then follows since $G/C_G(F(G))$ is isomorphic to a subgroup of $\prod_{p \in \BP} G/C_G(\op)$.  
\end{Remark}

The main step is the following:

\begin{lemma}\label{lemma:triviallayer}
	Let $G$ be a 2-tuple regular finite group with $E(G) = 1$. For $p \in \BP$, the group $G/C_G(\op)$ is solvable unless $p = 5$, $O_5(G) \cong C_5^2$ and $G/C_G(O_5(G)) \cong \SL(2,5)$. 
\end{lemma}

\begin{proof}
Fix $p \in \BP$, and suppose that $B \coloneqq G/C_G(\op)$ is non-solvable. By Lemma~\ref{lemma:oqfo}, the group $\op$ is 2-tuple regular. If $\op$ is non-abelian, we have $p = 2$ and $O_2(G) \in \{Q_8, G_{64}\}$ (see Theorem~\ref{theo:classificationpgroups}). Then $B$ is a subgroup of $\Aut(O_2(G))$ and hence solvable (see \cite[Lemma 4]{CHE00}). In the following, we therefore assume that $\op$ is abelian.
\medskip

In the following, we view $B$ as a subgroup of $\Aut(\op)$. Set $\Omega \coloneqq \Omega(\op)$. The kernel of the restriction map $B \to B|_{\Omega}$ is a $p$-group by \cite[Theorem 5.2.4]{GOR68} and hence solvable. 
Thus it remains to study the solvability of $A \coloneqq B|_{\Omega}$. By 2-tuple regularity, either all nontrivial elements in $\Omega$ are conjugate in $G$, or we have $[\omega]_G \subseteq \langle \omega \rangle$ for all $\omega \in \Omega$. In the latter case, $A$ is a group of power automorphisms and hence solvable. 
From now on, we assume that all nontrivial elements in $\Omega$ are conjugate in $G$.
Since $A$ can be viewed as a transitive subgroup of $\GL(\Omega)$, it is isomorphic to one of the groups listed in Theorem~\ref{theo:hering}.
\medskip

\emph{Observation:} Let $a_1, a_2 \in A$ with $\ord(a_1) = \ord(a_2)$. By Lemma~\ref{lemma:preimageorder}, we find preimages $g_1, g_2 \in G$ with $\ord(g_1) = \ord(g_2)$ of $a_1$ and $a_2$, respectively. Consider the assignment $g_1 \mapsto g_2$ and let $\Psi \colon G \to G$ be an extending bijection. If $g_1 x g_1^{-1} = x^\lambda$ holds for some $\lambda \in \Z$, we have $g_2 \Psi(x) g_2^{-1} = \Psi(x)^\lambda$, and vice versa. Interpreting $a_1$ and $a_2$ as elements of $\GL(\Omega)$, this means that $\Psi$ induces a bijection between the sets of eigenvectors of $a_1$ and $a_2$. In particular, $a_1$ and $a_2$ have the same eigenvalues with the same multiplicities in~$\F_p$.
%
%
%
\medskip

We now go through the list in Theorem~\ref{theo:hering}. Again, we write $\Omega_L$ for $\Omega$, viewed as an $L$-vector space. 
\medskip

\emph{Cases (a) and (b):} For $n^\ast = 1$, the group $A$ is solvable. For $n^\ast > 2$, there exists an involution in $A$ that acts as a power automorphism on $\Omega_L$ and others that do not, which contradicts the above observation. In the following, let $n^\ast = 2$, and let $N \trianglelefteq A$ be isomorphic to $\SL(\Omega_L) \cong \operatorname{Sp}(\Omega_L)$ (see \cite[Theorem II.9.12]{HUP67}). 
\medskip

First let $|L| \geq 7$. There exist elements $a,b \in N$ of order $|L|+1$ with distinct characteristic polynomials (when interpreted as elements of $\GL(\Omega_L)$). Let $g_a, g_b  \in G$ be preimages of $a$ and $b$ with $\ord(g_a) = \ord(g_b)$, respectively (see Lemma~\ref{lemma:preimageorder}). Let $x,y \in \Omega$ such that $\{x,y\}$ is an $L$-basis for $\Omega_L$. Let $x' \in \Omega$ such that $g_a \mapsto g_b$, $x \mapsto x'$ defines an isomorphism between $\langle g_a, x \rangle$ and $\langle g_b , x' \rangle$ and let $\Psi \colon G \to G$ be an extending bijection. Then the matrix describing the action of~$b$ on $\Omega_L$ with respect to the basis $\{x',\Psi(y)\}$ coincides with the matrix representing the action of $a$ on $\Omega_L$ with respect to the basis $\{x,y\}$. In particular, $a$ and $b$ have the same characteristic polynomial, which is a contradiction. 
\medskip

Now let $|L| = 5$. Then the involution in~$N$ acts on $\Omega$ by inversion. 
By the above observation on eigenspaces of elements in $A$, the group $A$ contains a single involution. Since $A \subseteq \operatorname{\Gamma L}(2,5) = \GL(2,5)$ holds and we have $|\GL(2,5) : \SL(2,5)| = 4$, we obtain $A = N \cong \SL(2,5)$. 
Note that $A$ contains precisely $20$ elements of order~3. If $O_5(G) \neq \Omega$ holds, we find pairs $(x,y), (x',y') \in O_5(G)^2$ satisfying $\langle x,y \rangle \cong \langle x', y' \rangle \cong C_{25}^2$ such that $x$ and $y$ are conjugate by an element of order $3$, whereas $x'$ and $y'$ are not. By Lemmas~\ref{lemma:preimageorder} and~\ref{lemma:commutators2tupleregular}, this is a contradiction. Hence $O_5(G) \cong C_5^2$ follows.
\medskip

\emph{Case (c):} Let $N \trianglelefteq A$ be isomorphic to $G_2(2^m)$. Then $N$ contains two conjugacy classes of involutions of different size (see \cite[Theorem 18.2]{ASC76}). As these classes are not fused in $A$, we obtain a contradiction to Lemma~\ref{lemma:centralizersq}. 
\medskip

\emph{Case (d):} These groups are solvable. 
\medskip

\emph{Case (e):} First consider the case (E1). Assume that $x \in A \setminus A^{(\infty)}$ is an element of order 3. By 2-tuple regularity, there exists $y \in A$ such that $X \coloneqq \langle x, y\rangle$ is isomorphic to $\SL(2,5)$. Since $X$ is perfect, we find $X \subseteq A^{(\infty)}$, which is a contradiction. Hence $A$ contains precisely $20$ elements of order~3. Due to $|\Omega| \geq 81$, there exist pairs $(x,y), (x',y') \in \Omega_L^2$ spanning $\Omega_L$ such that $x$ and $y$ are conjugate by an element of order~3 in $A$, whereas $x'$ and $y'$ are not. By Lemmas~\ref{lemma:preimageorder} and~\ref{lemma:commutators2tupleregular}, this is a contradiction. Similarly, we argue in the case (E4). 
For the case (E3), note first that all subgroups of $\GL(4,2)$ isomorphic to $A_7$ are conjugate. In these subgroups, some elements of order 3 have eigenvalue 1, whereas others do not. This contradicts the observation made above. A similar argument applies in the case (E2). In case (E5), we use that $\PSU(3,3)$ contains conjugacy classes of elements of order 3 of different sizes, which is a contradiction to Lemma~\ref{lemma:centralizersq}.
\medskip

Summarizing, this shows that $A$ is solvable unless $p = 5$, $O_5(G) \cong C_5^2$ and $G/C_G(O_5(G)) \cong \SL(2,5)$ holds. This finishes the proof.
\end{proof}

Having proven this result, the remaining part of the proof can be carried out as in \cite{CHE00}. For convenience of the reader, we summarize their arguments here. 
\begin{lemma}\label{lemma:cherlin7}
Let $G$ be a 2-tuple regular finite group with $E(G) = 1$ for which $C_G(O_5(G)) \subseteq O_5(G) Z(G)$ holds. Then $G$ is solvable.
\end{lemma}

\begin{proof}
Assume that $G$ is not solvable. By Lemma~\ref{lemma:triviallayer}, we have $O_5(G) \cong C_5^2$ and $G/C_G(O_5(G)) \cong \SL(2,5)$. Let $a \in G$ be a 2-element that acts on $O_5(G)$ by inversion. As in the proof of \cite[Lemma 7]{CHE00}, one shows that $G = O_5(G) \times C_G(a)$ holds. Since $|G/C_G(O_5(G))|$ is divisible by $5$, the group $C_G(a)$ contains an element of order 5, which is a contradiction to Lemma~\ref{lemma:oqfo}. 
\end{proof}

With this result, we can now prove the main result of this subsection:

\begin{theorem}\label{theo:triviallayer}
Let $G$ be a 2-tuple regular finite group with $E(G) = 1$. Then $G$ is solvable. 
\end{theorem}

\begin{proof}
By Lemmas~\ref{lemma:centralizerunionorderclasses} and~\ref{lemma:oqfo}, the group $K \coloneqq C_G(O_5(G))$ is 2-tuple regular. Due to $E(K) = 1$ and $O_5(K) \subseteq Z(K)$, the group $K$ is solvable by Remark~\ref{rem:solvablequotient} and Lemma~\ref{lemma:triviallayer}. By Theorem~\ref{theo:solvable}, we have $K = O_5(G) \times H$ for some $5'$-group $H$. By Corollary~\ref{cor:directproduct}, $H$ is 2-tuple regular. We first show that $G/C_G(F(H))$ is solvable. As in Remark~\ref{rem:solvablequotient}, it suffices to show that $G/C_G(O_p(H))$ is solvable for all $p \in \BP$. For $p = 5$, this is clear due to $O_5(H) = 1$. For $p \neq 5$, the group $G/C_G(O_p(H))$ is isomorphic to a quotient group of $G/C_G(O_p(G))$ and hence solvable by Lemma~\ref{lemma:triviallayer}. By Theorem~\ref{theo:solvable}, the group $\Aut(H/F(H))$ is solvable, which implies that $C_G(F(H))/C_G(H)$ is solvable. Thus it remains to show the solvability of $M \coloneqq C_G(H)$. Note that $M \in \ug$ follows by Lemma~\ref{lemma:centralizerunionorderclasses} and hence $M$ is 2-tuple regular. Moreover, we have $C_M(O_5(M)) = C_G(O_5(G)) \cap M = K \cap M = O_5(G) \times Z(H) \subseteq O_5(M) Z(M)$. By Lemma~\ref{lemma:cherlin7}, the group $M$ is solvable.   
\end{proof}

\subsection{Nontrivial layer}\label{sec:nontriviallayer}

Throughout this subsection, let $G$ be a $2$-tuple regular finite group with $E(G) \neq 1$. Our first aim is to show that $E(G)$ is quasisimple. To this end, we need the following observation on conjugate commuting involutions:

\begin{lemma}\label{lemma:simple}\label{lemma:quasisimple}
Let $S$ be a finite quasisimple group that is neither isomorphic to $\SL(2,q)$, where $q$ is a prime power, nor to~$2.A_7$ or $6.A_7$. Then $S$ contains a pair of conjugate commuting involutions. 
\end{lemma}

\begin{proof}
Let $Z^\ast(S)$ denote the preimage of $Z(S/O_{2'}(S))$ in $S$. First assume that there exists an involution $x \notin Z(S)$. Assume $x \in Z^\ast(S)$, so   $[x,S] \subseteq O_{2'}(S)$. Due to $O_{2'}(S) \subseteq Z(S)$, the element $x Z(S) \in S/Z(S)$ is a nontrivial central element, which is a contradiction to $S$ being quasisimple. Hence we have $x \notin Z^\ast(S)$.  
Let~$P$ be a Sylow $2$-subgroup of $S$ containing~$x$. By Glauberman's $Z^\ast$-theorem, we have $[x] \cap C_P(x) \neq \{x\}$ and the claim follows. Now assume that all involutions of $S$ are contained in $Z(S)$. Then $\bar{S} \coloneqq S/O_{2'}(S)$ is a quasisimple group with $O_{2'}(\bar{S}) = 1$. Due to $O_{2'}(S) \subseteq Z(S)$, all involutions in $\bar{S}$ are central. By \cite[Theorem]{GRI78}, we have $\bar{S} \cong 2.A_7$ or $\bar{S} \cong \SL(2,q)$ for some prime power $q \geq 5$, which implies that $S$ is isomorphic to $2.A_7$, to $6.A_7$ or to $\SL(2,q)$, respectively. This is a contradiction. 
\end{proof}

\begin{lemma}
The group $E(G)$ is quasisimple.
\end{lemma}

\begin{proof}
We begin with the following statement on quasisimple subnormal subgroups of $G$:
\medskip

\emph{Observation}: Let $q,y \in G$ be elements of the same order and suppose that there exists a quasisimple subnormal subgroup $Q \trianglelefteq \trianglelefteq G$ with $q \in Q \setminus Z(Q)$. We claim that there exists a subgroup $Y \trianglelefteq \trianglelefteq G$ isomorphic to $Q$ such that $y \in Y \setminus Z(Y)$ holds. Since $Q$ is quasisimple, we find $x \in Q$ with $\langle x, q \rangle = Q$. Now by 2-tuple regularity, there exists $x' \in G$ such that the assignment $x \mapsto x'$, $q \mapsto y$ defines an isomorphism between $Q$ and $Y \coloneqq \langle x', y \rangle$. Let $\Psi \colon G \to G$ be an extending bijection. For every $g \in G$, the group $Q$ is subnormal in $\langle Q, gQg^{-1} \rangle$ (see \cite[Theorem 1.2.8]{KUR98}). Since the assignment $x \mapsto x', q \mapsto y, g \mapsto \Psi(g)$ defines an isomorphism between $\langle Q, g\rangle$ and $\langle Y, \Psi(g) \rangle$, the group $Y$ is subnormal in $\langle Y, \Psi(g)Y\Psi(g)^{-1}\rangle$. Then \cite[Theorem 6.7.4]{KUR98} yields $Y \trianglelefteq \trianglelefteq G$.
\medskip

Now suppose that $M$ and $N$ are distinct components of $G$. Suppose that there exist elements $m_1, m_2 \in M$ that are conjugate in $M$ such that $\langle m_1, m_2 \rangle \cong C_p^2$ holds for some $p \in \BP$ dividing $|N|$. In particular, we have $m_1, m_2 \notin Z(E(G))$. Let $n \in N$ be an element of order $p$. The above observation yields $n \notin Z(E(G))$ and hence $\langle m_1, n \rangle \cong C_p^2$ follows. Then $\langle m_1, n\rangle \cong \langle m_1, m_2\rangle \cong C_p^2$ implies that $m_1$ and $n$ are conjugate in $G$ (see Lemma~\ref{lemma:commutators2tupleregular}). Due to $m_1, n \notin Z(E(G))$, it follows that $M$ and $N$ are conjugate. 
\medskip

Now let $K_1, \ldots, K_s$ denote the distinct components of $G$ and suppose that $s > 1$ holds. Assume that for some $i \in \{1, \ldots, s\}$, the component $K_i$ contains a pair of conjugate commuting involutions. By the above observation, this implies that all components of $G$ are conjugate. 
For $r \in \{1, \ldots, s\}$, we choose an involution $k_r \in K_r $. Then $k \coloneqq k_1 \cdots k_s$ is an involution. Using that $E(G)/Z(E(G))$ is a direct product of simple groups, it is easily seen that $k$ is not contained in any component of $G$, which contradicts the observation made at the beginning of this proof. 
%
\medskip

By Lemma~\ref{lemma:simple}, every component of $G$ is thus isomorphic to $2.A_7$, to $6.A_7$, or of the form $\SL(2,q)$ for some prime power $q \geq 5$. If there exists a component $K$ isomorphic to $2.A_7$ or $6. A_7$, then $K$ contains a pair of conjugate commuting elements of order 3. Since all components have order divisible by 3, the above argument can be used to show that they are conjugate, which again leads to a contradiction.
Hence every component is of the form $\SL(2,q)$ for some prime power $q$.
These groups all contain noncentral elements of order~3. Choosing such an element $k_r \in K_r \setminus Z(K_r)$ for every $r \in \{1, \ldots, s\}$, we can argue as before that $k \coloneqq k_1 \cdots k_s$ is an element of order 3 which is not contained in any component of $G$. This is a contradiction. Hence $s = 1$ follows, so $E(G)$ is quasisimple.
\end{proof}

\begin{lemma}\label{lemma:egsimpletr}
We have $E(G) \in \ug$. In particular, $E(G)$ is 2-tuple regular. 
\end{lemma}
\begin{proof}
First let $x \in E(G) \setminus Z(E(G))$, and consider $y \in G$ with $\ord(y) = \ord(x)$. Since $E(G)$ is quasisimple, there exists $x' \in E(G)$ with $\langle x, x' \rangle = E(G)$. By 1-tuple regularity, there exists $y' \in G$ such that the assignment $x \mapsto y$, $x' \mapsto y'$ defines an isomorphism between $E(G)$ and $E' \coloneqq \langle y, y' \rangle$. Let $\Psi \colon G \to G$ be an extending bijection. For every $g \in G$, we have $g E(G) g^{-1} \subseteq E(G)$, which implies $\Psi(g) E' \Psi(g)^{-1} \subseteq E'$. Hence $E' \trianglelefteq G$, so $E' \subseteq E(G)$ follows. In particular, we have $y \in E(G)$. Now let $x \in Z(E(G))$, and consider $y \in G$ with $\ord(y) = \ord(x)$. 
Since $E(G)$ is quasisimple, we find $x_1, x_2 \in E(G) \setminus Z(E(G))$ with $E(G) = \langle x_1, x_2 \rangle$. By 2-tuple regularity, there exist $y_1, y_2 \in G$ such that the assignment $x \mapsto y$, $x_1 \mapsto y_1$, $x_2 \mapsto y_2$ defines an isomorphism between $E(G) = \langle x, x_1, x_2 \rangle$ and $\langle y, y_1, y_2 \rangle$. In particular, we have $\ord(x_1) = \ord(y_1)$ and $\ord(x_2) = \ord(y_2)$. By the first part of this proof, this yields $y_1, y_2 \in E(G)$, and hence $\langle y, y_1, y_2 \rangle = E(G)$ follows. This implies $y \in E(G)$. 
\end{proof}

Now we determine the possibilities for the isomorphism type of $G$. As a preparation, we show that certain 1-ultrahomogeneous groups are not 2-tuple regular: 

\begin{Remark}\label{rem:extupleregular}
	Consider the group $G \coloneqq \PSL(2,8) \cong \langle (3,8,6,4,9,7,5), (1,2,3)(4,7,5)(6,9,8)\rangle \leq S_9$ as well as the elements $g \coloneqq (2,3)(4,6)(5,9)(7,8)$, $h_1 \coloneqq (2,4)(3,6)(5,7)(8,9)$ and $h_2 \coloneqq (2,6)(3,4)(5,8)(7,9)$ of $G$. We have $\langle g, h_1 \rangle \cong \langle g, h_2 \rangle \cong C_2 \times C_2$. However, it easily verified in GAP that there is no extending bijection for the assignment $g \mapsto g,\ h_1 \mapsto h_2$. Hence $G$ is not 2-tuple regular. Similarly, one checks that the groups $\PSL(2,9)$, $\SL(2,9)$ and $\PSL(3,4)$ are not 2-tuple regular. 
\end{Remark}

With this, we obtain the following result for the structure of $E(G)$: 

\begin{lemma}\label{lemma:posseg}
The group $E(G)$ is isomorphic to one of $\PSL(2,5)$, $\PSL(2,7)$, and $\SL(2,5)$.
\end{lemma}

\begin{proof}
By Lemma~\ref{lemma:egsimpletr}, the group $E(G)$ is 2-tuple regular. Since $E(G)$ is quasisimple, it is 2-generated and hence 1-ultrahomogeneous by Lemma~\ref{lemma:kgenerated}. By \cite[Corollary 2.4]{LI98}, $E(G)$ is isomorphic to one of $\PSL(2,5)$, $\PSL(2,7)$, $\PSL(2,8)$, $\PSL(2,9)$, $\PSL(3,4)$, $\SL(2,5)$ or $\SL(2,9)$ (the group $\SL(2,7)$ is mistakenly contained in their list as it contains elements of order~8 that are non-conjugate in its automorphism group, so it is not ultrahomogeneous). By Remark~\ref{rem:extupleregular}, the group $E(G)$ is isomorphic to $\PSL(2,5)$, $\PSL(2,7))$ or $\SL(2,5)$.
\end{proof}

We conclude with the following decomposition, whose proof is similar to that of \cite[Proposition 1]{CHE00}.

\begin{theorem}\label{theo:nonsolvable}
We have $G = H_0 \times E(G)$ for a 2-tuple regular solvable subgroup $H_0$ of $G$ whose order is coprime to $|E(G)|$.
\end{theorem}

\begin{proof}
Set $E \coloneqq E(G)$, and let $H \coloneqq C_G(E)$. By Lemma~\ref{lemma:posseg}, the group $E$ is isomorphic to one of $\SL(2,5)$, $\PSL(2,5)$, and $\PSL(2,7)$. By Lemma~\ref{lemma:egsimpletr}, we have $E \in \ug$. In particular, it contains all involutions in $G$.
\medskip

For a contradiction, suppose that $HE$ is a proper subgroup of $G$. Then there exists an element $g \in G$ inducing an outer automorphism of $E$. Since we have $|\Out(E)| = 2$ in all cases, we may assume that $g$ is a 2-element. We have $g^2 \in H$ and $\ord(g) \geq 4$ due to $g \notin E$. Then the involution $g^{\ord(g)/2}$ is contained in $H \cap E$. In particular, this forces $E \cong \SL(2,5)$. As $E$ then contains all elements of order $4$ in $G$, we find $\ord(g) \geq 8$. The element $h \coloneqq g^{\ord(g)/4} \in H$ has order 4 and hence $h \in E$ follows, which is a contradiction. Hence $G = H E$ follows. 
\medskip

If $H \cap E = 1$ holds, we set $H_0 \coloneqq H$. Otherwise, we have $E \cong \SL(2,5)$ and $H \cap E = Z(E)$ has order~2. Then $Z(E) \in \Syl_2(H)$ follows as all elements of order 4 of $G$ are contained in $E \setminus Z(E)$. This yields $H = O_{2'}(H) \times Z(E)$ and hence $G = H_0 \times E$ follows for $H_0 \coloneqq O_{2'}(H)$. In any case, $|H_0|$ and~$|E|$ are coprime. By Corollary~\ref{cor:directproduct}, $H_0$ is 2-tuple regular. Due to $E(H_0) = 1$, the group~$H_0$ is solvable by Theorem~\ref{theo:triviallayer}.
\end{proof}

\subsection{\texorpdfstring{Proof of Theorems~\ref{theo:a} and~\ref{theo:b}}{Proof Theorems A and B}}\label{sec:mainproofs}

We conclude this paper with the proof of our main theorems. 

\begin{proof}[Proof of Theorems~\ref{theo:a} and~\ref{theo:b}]
First assume that $G$ is a 2-tuple regular finite group. If $G$ is solvable, then $G$ is one of the groups described in Theorem~\ref{theo:solvable}. If $G$ is non-solvable, we have $G = H_0 \times E(G)$ by Theorem~\ref{theo:nonsolvable}, where $\gcd(|H_0|, |E(G)|) = 1$ and $H_0$ is a 2-tuple regular solvable group. The possibilities for $E(G)$ are given by Lemma~\ref{lemma:posseg}. Since $|E(G)|$ is even, the group $H_0$ is of odd order and hence abelian with homocyclic Sylow subgroups (see Theorem~\ref{theo:solvable}). 
\medskip 

Conversely, using the classification of the ultrahomogeneous finite groups stated in $\cite{CHE00}$, we see that all groups described in Theorem~\ref{theo:a} are ultrahomogeneous. In particular, they are indeed 2-tuple regular. This completes the proof of Theorem~\ref{theo:a}. Moreover, this shows that every 2-tuple regular finite group is ultrahomogeneous and hence $k$-ultrahomogeneous as well as $\ell$-tuple regular for every $k, \ell \in \N$. This proves Theorem~\ref{theo:b}. 
\end{proof}

\subsection*{Acknowledgments}
The research leading to these results has received funding from the European Research Council (ERC) under the European Unions Horizon 2020 research and innovation programme (EngageS: grant agreement No.~820148).
I thank J.~Brachter for helpful discussions as well as several GAP computations, and T.~Fritzsche and P.~Schweitzer for their comments on a preliminary version of this paper. Moreover, I am very grateful for the detailed comments and suggestions of the anonymous referee.

\bibliographystyle{plain}
\bibliography{regularity.bib}

\begin{thebibliography}{10}

\bibitem{ASC76}
M.~Aschbacher and G.~Seitz.
\newblock Involutions in {C}hevalley groups over fields of even order.
\newblock {\em Nagoya Math. J.}, 63:1--91, 1976.

\bibitem{BRA20}
J.~Brachter and P.~Schweitzer.
\newblock On the {W}eisfeiler-{L}eman {D}imension of {F}inite {G}roups.
\newblock In H.~Hermanns, L.~Zhang, N.~Kobayashi, and D.~Miller, editors, {\em
  Proceedings of the 35th Annual ACM/IEEE Symposium on Logic in Computer
  Science}, LICS '20, page 287–300. ACM, New York, 2020.

\bibitem{CAM80}
P.~Cameron.
\newblock 6-transitive graphs.
\newblock {\em J. Comb. Theory, Series B}, 28(2):168--179, 1980.

\bibitem{CHE91}
G.~Cherlin and U.~Felgner.
\newblock {Homogeneous Solvable Groups}.
\newblock {\em J. London Math. Soc.}, s2-44(1):102--120, 1991.

\bibitem{CHE00}
G.~Cherlin and U.~Felgner.
\newblock Homogeneous finite groups.
\newblock {\em J. London Math. Soc.}, 62(3):784--794, 2000.

\bibitem{GAP4}
The GAP~Group.
\newblock {\em {GAP -- Groups, Algorithms, and Programming, Version 4.11.1}},
  2021.

\bibitem{GOR68}
D.~Gorenstein.
\newblock {\em Finite Groups}.
\newblock Harper \& Row, New York, 1968.

\bibitem{GRI78}
R.~Griess.
\newblock Finite groups whose involutions lie in the center.
\newblock {\em Quart.~J.~Math.~Oxford}, 29(2):241--247, 1978.

\bibitem{HER74}
C.~Hering.
\newblock Transitive linear groups and linear groups which contain irreducible
  subgroups of prime order.
\newblock {\em Geometriae Dedicata}, 2:425--460, 1974.

\bibitem{HUP67}
B.~Huppert.
\newblock {\em Endliche {G}ruppen {I}}, volume 134 of {\em Grundlehren der
  mathematischen Wissenschaften}.
\newblock Springer, Berlin, 1967.

\bibitem{KUR98}
H.~Kurzweil and B.~Stellmacher.
\newblock {\em Theorie der endlichen {G}ruppen}.
\newblock Springer, Berlin, 1998.

\bibitem{Li99}
C.~Li.
\newblock A complete classification of finite homogeneous groups.
\newblock {\em Bull. Austral. Math. Soc.}, 60:331--334, 1999.

\bibitem{LI98}
C.~Li, C.~Praeger, and M.~Xu.
\newblock Isomorphisms of {F}inite {C}ayley {D}igraphs of {B}ounded {V}alency.
\newblock {\em J. Comb. Theory, Series B}, 73:164--183, 1998.

\bibitem{MAN71}
A.~Mann.
\newblock Generators of 2-groups.
\newblock {\em Israel J. Math.}, 10:158--159, 1971.

\bibitem{MIH04}
P.~Mih\u{a}ilescu.
\newblock Primary cyclotomic units and a proof of {C}atalans conjecture.
\newblock {\em J. reine angew. Mathematik}, 572:167--195, 2004.

\bibitem{ZHA92}
J.~Zhang.
\newblock On finite groups all of whose elements of the same order are
  conjugate in their automorphism groups.
\newblock {\em J. Algebra}, 153(1):22--36, 1992.

\end{thebibliography}

\end{document}